\documentclass[12pt]{amsart}

\author{Paul \textsc{Poncet}}
\address{CMAP, \'{E}cole Polytechnique, Route de Saclay, 91128 Palaiseau Cedex, France} 

\email{poncet@cmap.polytechnique.fr}

\usepackage[colorinlistoftodos,bordercolor=orange,backgroundcolor=orange!20,linecolor=orange,textsize=scriptsize]{todonotes}
\usepackage{stmaryrd}
\usepackage{amssymb}
\usepackage{amsmath}
\usepackage{graphicx}
\usepackage{t1enc}
\usepackage{color}
\usepackage[latin1]{inputenc}

\usepackage{calrsfs} 

\usepackage{times} 
\usepackage{bbm}
\usepackage{ifpdf}
\usepackage{tikz}
\usepackage{tikz-cd}
\usetikzlibrary{arrows}

\newcommand{\cis}{\Bbbk} 

\newcommand{\plus}{+}
\newcommand{\Plus}{\sum}

\DeclareMathOperator*{\val}{val}
\DeclareMathOperator*{\supp}{supp}
\DeclareMathOperator*{\divi}{div}

\newtheorem*{theorem*}{Theorem}
\newtheorem*{corollary*}{Corollary}

\newtheorem{theorem}{Theorem}[section]
\newtheorem{corollary}[theorem]{Corollary}
\newtheorem{proposition}[theorem]{Proposition}
\newtheorem{lemma}[theorem]{Lemma}

\theoremstyle{definition}
\newtheorem{definition}[theorem]{Definition}
\newtheorem{example}[theorem]{Example}

\newtheorem{remark}[theorem]{Remark}

\newtheorem*{acknowledgements*}{Acknowledgements}

\begin{document}

\title{Polynomials over idempotent semifields}

\date{\today}



\begin{abstract}
We study univariate polynomials with coefficients in an idempotent semifield and their factorization. 
We do not assume the idempotent semifield under consideration to be totally ordered, in contrast with most of the existing work on this topic. 
We notably determine when a polynomial splits into linear factors, and when its associated polynomial function does so. 
These results lead us to characterize algebraically closed idempotent semifields---those in which every polynomial function splits.  
We prove in particular that every complete idempotent semifield is algebraically closed.
We also relate algebraic closedness to the properties of preradicability and radicability and to the existence of solutions to polynomial equations or inequalities. 
\end{abstract}

\maketitle

\section{Introduction}

\subsection{General context and motivations}

Polynomials with coefficients in a commutative ring $A$ can have a non-unique factorization. 
For instance, given $A = \mathbb{Z}/n^2 \mathbb{Z}$, where $n$ is a natural number $> 1$, we have the identity  
\[
(X + n)^2 = X(X + 2n).   
\]
For polynomials with coefficients in an idempotent semifield $\cis$, one encounters the same situation. 
A semifield is a semiring $(\cis, \plus, 0, \times, 1)$ in which every non-zero element is invertible; it is an idempotent semifield if moreover $x + x = x$ for all $x \in \cis$. 
In an idempotent semifield $\cis$ we have
\[
(X \plus a)(X \plus b) = (X \plus a \vee b)(X \plus a \wedge b), 
\]
where $a \vee b := a + b$ is the sup of $\{a, b \}$ and $a \wedge b$ is the inf of $\{ a, b \}$. 
If $\cis$ is the \textit{max-plus algebra} (also called \textit{tropical semifield}), i.e.\ the idempotent semifield 
\[
\mathbb{R}_{\max} = (\mathbb{R} \cup \{-\infty \}, \max, -\infty, +, 0), 
\]
or more generally if $\cis$ is totally ordered, then $\{ a, b \} = \{ a \vee b, a \wedge b \}$, and the factorization is unique. 
However, this is not the case in general. 

In the setting of totally ordered idempotent semifields, such as the max-plus algebra, a rich factorization theory has been developed, with major contributions from Cuninghame-Green and Meijer \cite{Cuninghame80}, Baccelli et al.\ \cite{Baccelli92}, Akian, Bapat, and Gaubert \cite{Akian04}, Butkovic \cite[Section~5]{Butkovic10}, Castella \cite{Castella10, Castella13}, Hong and Sendra \cite{Hong18}. 
A striking feature of this theory is that many structural properties of polynomials---such as factorization into linear factors and the existence of roots---can be described in terms of order-theoretic and convexity properties. 
However, the assumption of total order is restrictive and is not preserved under natural constructions such as Cartesian products or quotients.

The purpose of this paper is to develop a theory of polynomial factorization over idempotent semifields without assuming total order.
In that regard, this work can be seen as a continuation of Rump \cite{Rump15}, \cite{Rump16}. 
In this general setting, several difficulties arise compared to polynomials over a field with characteristic $0$ or a totally ordered idempotent semifield, in addition to the loss of unique factorization. 
Most notably, the natural map that associates to a polynomial its polynomial function is no longer injective, so that formal polynomials cannot be identified with their associated functions. 
This leads to the coexistence of several distinct notions of roots and algebraic closedness.

\subsection{Main results}

To address these hurdles, we revisit notions imported from the totally ordered case. 
A polynomial $p$ is said to be \textit{closable} if a certain family of infima associated with its polynomial function exists; in this case, one can canonically associate to it a maximal representative $\overline{p}$ having the same polynomial function, with a construction reminiscent of the Legendre--Fenchel transform. 
A polynomial $p$ is \textit{closed} if it is closable and such that $p = \overline{p}$. 

We prove that closed polynomials coincide with polynomials that split into products of linear factors. 
More precisely, we obtain the following result, which originates from Cuninghame-Green and Meijer \cite[Theorem~8]{Cuninghame80}, who discovered the fundamental result of factorization of polynomials with coefficients in the tropical semifield $\mathbb{R}_{\max}$. 
It also extends Baccelli et al.\ (\cite[Lemma~3.41]{Baccelli92}, \cite[Theorem~3.43]{Baccelli92}) and Rump \cite[Proposition~3]{Rump15}. 

\begin{theorem*}[First fundamental theorem, simplified version]
Let $p \in \cis[X]$ be a polynomial of degree $n > 0$. 
Then $p$ is closed if and only if $p$ splits in $\cis$. 
In this case, 
\[
p(X) = p_n (X \plus c_1) \cdots (X \plus c_n), 
\] 
where $c_k := p_{n-k}^{} \cdot p_{n-k+1}^{-1}$ if $p_{n-k+1} \neq 0$, and $c_k := 0$ otherwise, and we have $c_1 \geqslant \ldots \geqslant c_n$. 
\end{theorem*}

The elements $c_k$ are called the \textit{corners} of the closed polynomial $p$. 
The second theorem characterizes closable polynomials as those whose associated polynomial functions admit such a factorization.
It extends Baccelli et al.\ \cite[Theorem~3.89]{Baccelli92}. 

\begin{theorem*}[Second fundamental theorem]
Let $p \in \cis[X]$ be a polynomial of degree $n > 0$. 
Then $p$ is closable if and only if $\widehat{p}$ splits in $\cis$. 
In this case, $\widehat{p}(x) = p_n (x \plus c_1) \cdots (x \plus c_n)$, for all $x \in \cis$, where $c_1 \geqslant \ldots \geqslant c_n$ are the corners of $\overline{p}$. 
\end{theorem*}

As a consequence, we obtain a characterization of algebraically closed idempotent semifields in terms of closability of polynomials. 

\begin{corollary*}
An idempotent semifield $\cis$ is algebraically closed if and only if every polynomial of $\cis[X]$ is closable. 
\end{corollary*}

Here, an idempotent semifield $\cis$ is \textit{algebraically closed} if, for every polynomial $p \in \cis[X]$ of degree $> 0$, the associated polynomial function $\widehat{p}$ splits in $\cis$. 
Other definitions associated to the term of ``algebraic closedness'' can be found in the literature. 
An idempotent semifield $\cis$ is \textit{radicable} (algebraically closed in the sense of Shpiz \cite{Shpiz00}) if, for all $y \in \cis$, $n \in \mathbb{N}^*$, there is an element (necessarily unique) $x \in \cis$ such that $x^n = y$.  
Shpiz proved that, in this situation, every polynomial equation of the form $\widehat{p}(x) = t$ has a (greatest) solution if and only if $p_0 \leqslant t$, and this solution is unique if $p_0 = 0$. 
More generally, we prove the existence of solutions to certain polynomial equations of the form $\widehat{p}(x) = \widehat{q}(x)$ under the assumption of radicability. 
We also introduce the weaker notion of \textit{preradicability}, which helps in proving the existence of solutions to polynomials inequalities of the form $\widehat{p}(x) \leqslant \widehat{q}(x)$. 

Also, an idempotent semifield $\cis$ is \textit{equationally closed} (algebraically closed in the sense of Rump \cite{Rump15}) if, for all $p, q \in \cis[X]$, the existence of a solution to $\widehat{p}(x) = \widehat{q}(x)$ in an extension of $\cis$ implies the existence of a solution in $\cis$ itself. 
Rump characterized equationally closed idempotent semifields as those idempotent semifields being radicable and having the property that, whenever $x \wedge y = 1$ and $x \plus y \leqslant z$, there exists $x' \geqslant x$ and $y' \geqslant y$ such that $x' \wedge y' = 1$ and $x' \plus y'  = z$.
Interestingly, this latter property always holds in the totally ordered case. 


Gathering results from the literature and results proved in this paper, we highlight the way these notions are linked with one another as follows. 

\begin{theorem*}
Let $\cis$ be an idempotent semifield. 
Consider the following conditions:
\begin{enumerate}
  \item\label{prop:star1} $\cis$ is equationally closed;
  \item\label{prop:star2} $\cis$ is radicable;
  \item\label{prop:star3} $\cis$ is algebraically closed and order-dense;
  \item\label{prop:star4} $\cis$ is algebraically closed;
  \item\label{prop:star5} $\cis$ is preradicable. 
\end{enumerate}
Then \eqref{prop:star1} $\Rightarrow$ \eqref{prop:star2} $\Rightarrow$ \eqref{prop:star3} $\Rightarrow$ \eqref{prop:star4} $\Rightarrow$ \eqref{prop:star5}. 
If moreover $\cis$ is totally ordered, then \eqref{prop:star1} $\Leftrightarrow$ \eqref{prop:star2} $\Leftrightarrow$ \eqref{prop:star3} $\Rightarrow$ \eqref{prop:star4} $\Leftrightarrow$ \eqref{prop:star5}. 
\end{theorem*}



\subsection{Structure of the article}




The paper is organized as follows. 
In Section~\ref{sec:semirings}, we recall usual definitions from poset theory and semiring theory, as well as some essential properties of idempotent semirings and idempotent semifields. 
In Section~\ref{sec:poly}, we recall the definition of a polynomial with coefficients in an idempotent semifield $\cis$. 
Thanks to a result due to Crosby, we prove that the set $\cis[X]$ of polynomials over $\cis$ has the structure of a residuated idempotent semiring without zero divisors. 
We then consider polynomial functions and show that they follow a series of interesting properties. 
Notably, every polynomial function preserves arbitrary sups of nonempty subsets, and arbitrary non-zero infs of nonempty subsets; remarkably, it is an order-embedding (hence is injective) whenever it maps $0$ to $0$. 
Section~\ref{sec:clo} is a cornerstone of the paper, as we introduce the notions of \textit{closable} and \textit{closed} polynomial. 
We prove that a polynomial is closable if and only if there is a greatest polynomial with the same polynomial function. 
Closable and closed polynomials reveal their importance in Section~\ref{sec:ft} with our first and second ``fundamental theorems'', where we prove that a polynomial is closed (resp.\ closable) if and only if it splits (resp.\ its polynomial function splits). 
Calling an idempotent semifield \textit{algebraically closed} if every polynomial function splits, we obtain the characterization that an idempotent semifield is algebraically closed if and only if every polynomial is closable. 
It is then a trivial consequence that every complete idempotent semifield, in particular $\mathbb{R}_{\max}$, is algebraically closed. 
In passing, we prove that a basket of roots of size $k < n$ of a closed polynomial of degree $n$ can be completed to a basket of roots of size $n$. 
Section~\ref{sec:ineq} focuses on the concept of a \textit{preradicable} idempotent semifield, a property that is shown to be useful for solving polynomial inequalities. 
We also prove that every algebraically closed idempotent semifield is preradicable, and that the converse statement holds true if the idempotent semifield is totally ordered. 
In Section~\ref{sec:eq}, we introduce \textit{radicability} as a property stronger than preradicability to qualify an idempotent semifield. 
This property attests to its strength as it provides existence of solutions to polynomial equations. 
We also prove that every radicable idempotent semifield is algebraically closed and order-dense, and that the converse statement holds true if the idempotent semifield is totally ordered. 
In the last part of the section, we consider \textit{rational polynomials}, defined as formal polynomials modulo a non-zero polynomial factor. 
We also prove that, given a radicable idempotent semifield, there is an isomorphism of semirings between rational polynomials and polynomial functions, and that every rational polynomial splits. 
The final section discusses perspectives and further developments.

\subsection{General notations}

We write $\mathbb{R}$ (resp.\ $\mathbb{R}_+$) for the set of real numbers (resp.\ nonnegative real numbers). 
We also write $\mathbb{Q}$ (resp.\ $\mathbb{Q}_+$, $\mathbb{Z}$, $\mathbb{N}$) for the set of rational numbers (resp.\ non-negative rational numbers, integers, natural numbers) and $\mathbb{N}^*$ for the set of non-zero natural numbers.


\section{Semirings and semifields}\label{sec:semirings}

In this section, we recall usual definitions from poset theory and semiring theory, as well as some properties of idempotent semirings and idempotent semifields that will be used repeatedly along this paper. 

\subsection{Partially ordered sets and lattices}

A \textit{partially ordered set} or \textit{poset} $(P,\leqslant)$ is a set $P$ together with a reflexive, transitive, and antisymmetric binary relation $\leqslant$. 
An \textit{upper bound} of a subset $A$ is an element $u$ such that $x \leqslant u$ for all $x \in A$. 
An upper bound of $A$ is the \textit{least upper bound} or \textit{sup} of $A$, written $\bigvee A$, if $\bigvee A \leqslant u$ for every upper bound $u$ of $A$. 
The notions of \textit{lower bound} and \textit{inf} are defined dually; the inf of a subset $A$, if it exists, is written $\bigwedge A$. 
A \textit{lattice} is a poset in which every nonempty finite subset has both an inf and a sup; we write $x \vee y$ for $\bigvee \{ x, y \}$ and $x \wedge y$ for $\bigwedge \{ x, y \}$. 

\subsection{Definition of semirings and semifields}

A \textit{semiring} is a commutative monoid $(S, \plus, 0)$ endowed with an additional binary relation $\times$ (the multiplication) that is associative,  has an identity $1$, distributes over $\plus$, and admits $0$ as absorbing element, i.e.\ $0 x = x 0 = 0$ for all $x \in S$. 

A semiring $S$ is:
\begin{itemize}
  \item \textit{idempotent} if $x + x = x$, for all $x \in S$; 
  \item \textit{commutative} if $x y = y x$, for all $x, y \in S$;
  \item \textit{without zero divisors} if ${ 0 = x y } \Rightarrow { 0 \in \{ x, y \} }$, for all $x, y \in S$;
  \item \textit{$\times$-cancellative} if $(x z = y z \mbox{ or } z x = z y) \Rightarrow x = y$, for all $x, y \in S$ and $z \in S \setminus \{ 0 \}$. 
\end{itemize}
On semiring theory, the reader may refer to Golan's books \cite{Golan99} and \cite{Golan03}, Glazek's survey \cite{Glazek02}, and Gondran and Minoux's monograph \cite{Gondran08}. 

An idempotent semiring has a natural structure of partially ordered set with $x \leqslant y \Leftrightarrow  x \plus y = y$, whose least element is $0$. 
This makes $x \plus y$ the sup of the pair $\{x, y\}$, i.e.\ 
\[
x \plus y = x \vee y, 
\]
for all $x, y$. 
Therefore, given elements $x, y, z$, the reader should get used to assertions such as $(x \leqslant z \mbox{ and } y \leqslant z) \Leftrightarrow x + y \leqslant z$.  

A \textit{semifield} is a semiring with an identity $1 \neq 0$ in which every non-zero element has a multiplicative inverse. 
Idempotent semifields coincide with the class of semifields with characteristic one, where the \textit{characteristic} of a semifield is the least $k \in \mathbb{N}^*$, if it exists, such that $(k + 1) 1 = 1$ (if it does not exists, it is conventionally set to $0$). 
The concept of characteristic one has been introduced simultaneously by Connes and Consani \cite{Connes12} and Castella \cite{Castella10}; see also the early work of Zeleznekow \cite{Zeleznikow81a} on inverse semirings, and Rump \cite[Section~2]{Rump16}. 

An idempotent semiring or semifield is \textit{totally ordered} if $x \plus y \in \{x, y\}$, for all $x, y$; equivalently, if $x \leqslant y$ or $y \leqslant x$, for all $x, y$. 
Note that \textit{linearly ordered} is a classical synonym for totally ordered; one may also find in the literature terms like \textit{bipotent} (see \cite{Perri13}) or \textit{selective} (see \cite{Gondran08}, \cite{Tolliver14}). 

\begin{example}[Boolean semifield]
The \textit{Boolean semifield} $\mathbb{B}$ is defined as the set $\{ 0, 1 \}$ with $0 \neq 1$, endowed with the usual multiplication and the addition $+$ defined by $0 + 0 = 0$ and $1 + 0 = 0 + 1 = 1 + 1 = 1$. 
Recall that $\mathbb{B}$ is the only finite idempotent semifield \cite[Proposition~1.35]{Golan03}. 
\end{example}

\begin{example}[Tropical semifield]
The \textit{tropical semifield} (or \textit{max-plus algebra}) $\mathbb{R}_{\max}$ is defined as
\[
\mathbb{R}_{\max} := (\mathbb{R} \cup \{ -\infty \}, \max, -\infty, +, 0). 
\]
This idempotent semifield and its dual $\mathbb{R}_{\min} := (\mathbb{R} \cup \{ +\infty \}, \min, +\infty, +, 0)$ have been rediscovered many times due to their versatility in various areas of pure and applied mathematics, see e.g.\ Akian et al.\ \cite{Akian06} for background and references. 
\end{example}

\begin{example}\label{ex:qdiv}
Consider the set $\mathbb{Q}_+$ endowed with the usual multiplication and the dual divisibility relation $\preceq$, so that $x \preceq y$ if $y$ divides $x$, in the sense that $x = n y$, for some $n \in \mathbb{N}$. 
Then $0$ is the least element of $\mathbb{Q}_+$ with respect to $\preceq$, and $\preceq$ is a lattice order, where $x \vee y$ (resp.\ $x \wedge y$) is the greatest common divisor (resp.\ least common multiple) of $\{ x, y \}$.
This makes 
\[
\mathbb{Q}_+^{\divi} := (\mathbb{Q}_+, \vee, 0, \times, 1)
\]
into a (not totally ordered) idempotent semifield. 
\end{example}

\subsection{Some properties of idempotent semirings and semifields}\label{subsec:prop}

In this paragraph we recall some classical properties of idempotent semirings and idempotent semifields that will be needed many times in this paper. 

\begin{lemma}[Binomial identity]\label{lem:bi}
Let $S$ be a commutative idempotent semiring. 
Then 
\[
(x \plus y)^n = \Plus_{k=0}^n x^k y^{n-k}, 
\]
for all $x, y \in S$ and $n \in \mathbb{N}$. 
\end{lemma}

\begin{proof}
The result can easily be proved by induction on $n$. 
\end{proof}

\begin{lemma}[Frobenius identity]\label{lem:frobenius}
Let $S$ be a $\times$-cancellative, commutative idempotent semiring. 
Then 
\begin{align*}\label{eq:frobenius}
(x \plus y)^n = x^n \plus y^n, 
\end{align*}
for all $x, y \in S$ and $n \in \mathbb{N}$. 
\end{lemma}

\begin{proof}
See e.g.\ \cite[Lemma~4.3]{Connes10}. 
\end{proof}

\begin{lemma}\label{prop:perfore}
Let $S$ be a $\times$-cancellative, commutative idempotent semiring. 
Let $x, y \in S$. 
If $x^n \leqslant y^n$ for some $n \in \mathbb{N}^*$, then $x \leqslant y$. 
\end{lemma}

\begin{proof}
See e.g.\ \cite[Proposition~2.24]{Golan03}. 
\end{proof}

Given an idempotent semifield $\cis$, the set $\cis^* := \cis \setminus \{0\}$ of invertible elements is a \textit{lattice-ordered group} (or \textit{lattice-group} or \textit{$\ell$-group}). 
Lattice-ordered groups have been widely studied in the literature; we refer the reader to the monographs by Birkhoff \cite{Birkhoff67} and Bigard, Keimel, and Wolfenstein \cite{Bigard77}. 

An idempotent semifield $\cis$ inherits from many properties of the lattice-ordered group $\cis^*$. 
It is in particular a lattice and  
\begin{equation}\label{eq:cisinf}
x \wedge y = (x^{-1} \plus y^{-1})^{-1}, 
\end{equation}
for all $x, y \in \cis^*$. 
If $\cis$ is commutative, \eqref{eq:cisinf} reformulates as the \textit{modular identity} 
\begin{equation}\label{eq:modular}
x y = (x \vee y) (x \wedge y), 
\end{equation}
for all $x, y \in \cis$. 
Moreover, the following identities hold: 
\begin{equation}\label{eq:infsupf1}
x (\bigvee A) = \bigvee (x A), 
\end{equation}
for all $x$ and all subsets $A$ with sup, and  
\begin{equation}\label{eq:infsupf}
x (\bigwedge A) = \bigwedge (x A), 
\end{equation}
for all $x$ and all subsets $A$ with inf. 
We also have similar distributivity properties with $\plus$ in place of $\times$:
\begin{equation}\label{eq:infplus1}
x \plus (\bigvee A) = \bigvee (x \plus A), 
\end{equation}
for all $x$ and all \textit{nonempty} subsets $A$ with sup, and  
\begin{equation}\label{eq:infplus}
x \plus (\bigwedge A) = \bigwedge (x \plus A), 
\end{equation}
for all $x$ and all nonempty subsets $A$ with \textit{non-zero} inf, or all \textit{nonempty finite} subsets $A$. 

\begin{lemma}\label{prop:sn}
Let $\cis$ be a commutative idempotent semifield, and let $n \in \mathbb{N}$. 
Then  
\[
(\bigvee A)^n = \bigvee A^n, 
\]
for all subsets $A$ of $\cis$ with sup, and 
\[
(\bigwedge A)^n = \bigwedge A^n, 
\]
for all subsets $A$ of $\cis$ with inf, where $A^n := \{ a^n : a \in A \}$. 
\end{lemma}

\begin{proof}
We only prove the second assertion (the first assertion can be proved along similar lines). 
Let $A$ be a subset of $\cis$ with inf. 
If $n = 0$ the result is clear, so suppose that $n > 0$. 
We have that $(\bigwedge A)^n$ is a lower bound of $A^n$. 
Now let $\ell$ be another lower bound of $A^n$, and let $a_1, \ldots, a_n \in A$. 
Then 
\begin{align*}
a_1 \cdots a_n &\geqslant (a_1 \wedge \cdots \wedge a_n) \cdots (a_1 \wedge \cdots \wedge a_n) \\
&= (a_1 \wedge \cdots \wedge a_n)^n = a_1^n \wedge \cdots \wedge a_n^n \\
&\geqslant \ell \wedge \cdots \wedge \ell = \ell. 
\end{align*}
Using Equation~\eqref{eq:infsupf} repeatedly we deduce that $(\bigwedge A)^n \geqslant \ell$. 
This shows that $(\bigwedge A)^n$ is the greatest lower bound of $A^n$, i.e.\ its inf. 
\end{proof}

The following lemma is a slight generalization of Litvinov et al.\ \cite[Lemma~5.1]{Litvinov01}. 

\begin{lemma}\label{lem:inf0bis}
Let $\cis$ be an idempotent semifield. 
Then $\cis^*$ has an inf, and it satisfies $\bigwedge \cis^* = 0$ if and only if $\cis \neq \mathbb{B}$. 
\end{lemma}

\begin{proof}
Assume that $\cis \neq \mathbb{B}$, and let $t \in \cis^*, t \neq 1$. 
Let $\ell$ be a lower bound of $\cis^*$. 
If $\ell = 1$, then $t > 1$ by definition of $\ell$, hence $\cis^* \ni t^{-1} < 1 = \ell$, a contradiction. 
Thus, $\ell \neq 1$. 
Since $1 \in \cis^*$, we have $\ell < 1$ by definition of $\ell$. 
If $\ell \neq 0$, then $\ell^2 \in \cis^*$ and $\ell^2 < \ell$, a contradiction. 
We conclude that $\ell = 0$, which proves that $0$ is the inf of $\cis^*$. 
\end{proof}




\section{Polynomials and polynomial functions}\label{sec:poly}

In this section, we recall the definition of a polynomial with coefficients in an idempotent semifield $\cis$. 
Thanks to a result due to Crosby, we prove that the set $\cis[X]$ of polynomials over $\cis$ has the structure of a residuated idempotent semiring without zero divisors. 
We then consider polynomial functions and show that they follow a series of interesting properties. 
Notably, every polynomial function 
\begin{itemize}
  \item preserves arbitrary sups of nonempty subsets;
  \item preserves arbitrary non-zero infs of nonempty subsets;
  \item is a lattice morphism, so is additive in particular, and
  \item is an order-embedding (hence is injective) if it maps $0$ to $0$. 
\end{itemize}

\subsection{Disclaimer}

In the remaining part of this paper, all idempotent semirings and idempotent semifields considered are supposed to be commutative and infinite (hence different from $\mathbb{B}$). 
The letter $\cis$ will customarily denote a (commutative, infinite) idempotent semifield. 

\subsection{Definitions}

A \textit{formal polynomial}, or \textit{polynomial} for short, with coefficients in $\cis$ is a sequence $p = (p_0, p_1, \ldots)$ of elements of $\cis$ that is eventually zero. 
The \textit{support} $\supp p$ of $p$ is the (finite) set of all $j$ such that $p_j \neq 0$. 
The \textit{degree} $\deg p$ (resp.\ the \textit{valuation} $\val p$) of $p$ is the maximum element (resp.\ minimum element) of $\supp p$, with the conventions $\deg 0 = -\infty$ and $\val 0 = \infty$. 
The addition $p \plus q = p \vee q$ of two polynomials $p$ and $q$ is defined by $(p \plus q)_j = p_j \plus q_j$, and the multiplication $p \cdot q$ is given by the Cauchy product $(p \cdot q)_j = \Plus_{i=0}^j p_i q_{j-i}$. 

If we let $X$ denote the polynomial $(0, 1, 0, 0, \ldots)$, then every polynomial $p$ can be written as 
\[
p = \Plus_{j \in \mathbb{N}} p_j X^j. 
\]
We shall often use $p(X)$ as a synonym for $p$. 
The set $\cis[X]$ of polynomials with coefficients in $\cis$ has the structure of an idempotent semiring. 
This is also a lattice, the inf $p \wedge q$ of two polynomials $p, q$ satisfying $(p \wedge q)_j = p_j \wedge q_j$ for all $j$, and it is easily verified that the multiplication distributes over $\wedge$. 
Recall from \cite[Remark~3]{Castella10} that the multiplication in $\cis[X]$ is never cancellative, since, for instance,  
\[
(X \plus 1) (X^2 \plus 1) = (X \plus 1) (X^2 \plus X \plus 1). 
\]
In particular, $\cis[X]$ cannot be embedded into an idempotent semifield of fractions. 
Note also that we have the following submodular identity in $\cis[X]$: 
\[
(p \vee q) (p \wedge q) \leqslant p q. 
\]
However, equality does not hold in general (e.g., take $p = X$ and $q = 1$). 

\begin{remark}
The more general setting of polynomials over \textit{semirings} has been studied in the literature, see e.g.\ Hebisch and Weinert \cite{Hebisch98}, Nasehpour \cite{Nasehpour25}, Akian et al.\ \cite{Akian25}.
\end{remark}

\subsection{Derivative of a polynomial}

Let $p \in \cis[X]$. 
The \textit{derivative} of $p$ is the polynomial $p'$ in $\cis[X]$ defined by 
\[
p'(X) = \Plus_{j \geqslant 1} p_j X^{j-1} = \Plus_{j \geqslant 0} p_{j+1} X^{j}. 
\]
In other words, $p'_j = p_{j+1}$, for all $j \in \mathbb{N}$. 
Moreover, 
\[
p(X) = { X p'(X) + p_0 }.
\]
We shall sometimes write $p^{(1)}$ for the derivative $p'$ of $p$, and $p^{(i)}$ for the derivative of $p^{(i-1)}$. 

\begin{proposition}\label{prop:der}
Let $p, q \in \cis[X]$, and let $\lambda \in \cis$.  
Then 
\begin{itemize}
	\item $(\lambda p)' = \lambda p'$, 
	\item $(p \plus q)' = p' \plus q'$, 
	\item $(p \cdot q)' = (p' \cdot q) \plus (p \cdot q')$, 
\end{itemize}
\end{proposition}

\begin{proof}
The easy proof is left to the reader. 
\end{proof}

\subsection{Polynomial division}

We recall a result due to Crosby \cite[Theorem~5.2.3]{Crosby95}, who examined polynomials with coefficients in an idempotent semifield for applications to image processing and mathematical morphology. 

\begin{theorem}[Crosby]\label{thm:crosby}
Let $p, q \in \cis[X]$ with $q \neq 0$. 
Then there exists a polynomial in $\cis[X]$, denoted by $p / q$ and called the quotient of $p$ by $q$, such that 
\[
{ f \cdot q \leqslant p } \Leftrightarrow { f \leqslant p / q },
\]
for all $f \in \cis[X]$. 
Moreover, we have 
\begin{equation}\label{eq:crosby}
p / q = \bigwedge_{i \in \supp q} q_i^{-1} p_{}^{(i)}. 
\end{equation}
\end{theorem}

\begin{proof}
If $f \cdot q \leqslant p$, then $f \cdot q^{(i)} \leqslant p^{(i)}$, for all $i \in \{ 0, \ldots, \deg q \}$. 
Since $q_i \leqslant q^{(i)}$, this implies $f \leqslant q_{i}^{-1} p_{}^{(i)}$, for all $i \in \supp q$. 

Conversely, assume that $f \leqslant q_{i}^{-1} p_{}^{(i)}$, for all $i \in \supp q$. 
This implies in particular that $f \cdot q^{(i)} \leqslant p^{(i)}$ for $i = \deg q$. 
Now if  $f \cdot q^{(i)} \leqslant p^{(i)}$ for some $i > 0$, then 
\begin{align*}
f(X) \cdot q^{(i-1)}(X) &= f(X) (X q^{(i)}(X) \plus q_{i-1}) \\
&\leqslant X p^{(i)}(X) \plus p_{}^{(i-1)}(X) = p^{(i-1)}(X).
\end{align*}
By induction we get $f \cdot q^{(i)} \leqslant p^{(i)}$, for all $i \in \{ 0, \ldots, \deg q \}$. 
For $i = 0$ we obtain $f \cdot q \leqslant p$, as required.
\end{proof}

\begin{remark}
The definition of $p/q$ shows that $\deg (p/q) \leqslant \deg p - \deg q$. 
However, equality does not hold in general. 
For instance, if $p(X) = X^2 + 1$ and $q(X) = X + 1$, then $(p/q)(X) = 0$. 
\end{remark}

\begin{corollary}[Crosby]\label{coro:crosby}
Let $p, q \in \cis[X]$ with $q \neq 0$. 
Then $q$ is a factor of $p$ if and only if $p = (p/q) \cdot q$. 
\end{corollary}

\begin{proof}
Assume that $p = r \cdot q$, for some $r \in \cis[X]$. 
Then $r \leqslant p / q$, so $p = r \cdot q \leqslant (p/q) \cdot q \leqslant p$. 
This shows that $p = (p/q) \cdot q$. 
The converse statement is clear. 
\end{proof}

Given a poset $P$, recall that a map $f : P \rightarrow P$ is \textit{residuated} (see e.g.\ Blyth and Janowitz \cite[p.~11]{Blyth72}) if there exists some (necessarily unique) map $f^{\#}$ satisfying 
\[
f(x) \leqslant y \Leftrightarrow x \leqslant f^{\#}(y),
\]
for all $x, y \in P$. 
An idempotent semiring $S$ is \textit{residuated} if, for all $z \in S$ with $z \neq 0$, the map $f_z : S \to S, x \mapsto x z$ is residuated. 
The element $f_z^{\#}(y)$ is denoted by $y / z$, and we have by definition
\[
x z \leqslant y \Leftrightarrow x \leqslant y / z,
\]
for all $x, y, z \in S$ with $z \neq 0$. 

\begin{corollary}
The set $\cis[X]$ has the structure of a residuated idempotent semiring with no zero divisors. 
\end{corollary}

\begin{proof}
The fact that $\cis[X]$ is a residuated semiring is given by Crosby's theorem. 
Let us show that $\cis[X]$ has no zero divisors. 
To this end, let $p, q \in \cis[X]$ with $p \cdot q = 0$ and $q \neq 0$. 
Then $p \leqslant { 0 / q }$, and $0 / q$ is zero by Equation~\eqref{eq:crosby}, so $p = 0$. 
\end{proof}

\subsection{Polynomial functions}

To every polynomial $p \in \cis[X]$ we can associate the \textit{polynomial function} $\widehat{p} : \cis \rightarrow \cis$ defined by 
\[
\widehat{p}(x) = \Plus_{j \in \mathbb{N}} p_j x^j,  
\]
for all $x \in \cis$. 
The notation $\widehat{p}$ is borrowed from Baccelli et al.\ \cite{Baccelli92}. 
The mapping $p \mapsto \widehat{p}$ is a semiring morphism from $\cis[X]$ to the semiring of polynomial functions \cite[Lemma~3.35]{Baccelli92}. 
Note that this mapping is not injective in general, so there is no one-to-one correspondence between polynomials and polynomial functions; for instance, the polynomial functions associated with the polynomials $X^2 \plus 1$ and $(X + 1)^2$ coincide by Lemma~\ref{lem:frobenius}. 

Every polynomial function $\widehat{p}$ 
satisfies the following \textit{convexity property}. 

\begin{proposition}\label{prop:convex}
Let $p \in \cis[X]$. 
Then we have 
\begin{align*}
\widehat{p}(z)^{m_1 + m_2} \leqslant \widehat{p}(x)^{m_1} \widehat{p}(y)^{m_2}, 
\end{align*}
for all $x, y, z \in \cis$ and $m_1, m_2 \in \mathbb{N}$ such that $z^{m_1 + m_2} = x^{m_1} y^{m_2}$. 
\end{proposition}

\begin{proof}
Using twice the Frobenius identity of Lemma~\ref{lem:frobenius}, we have 
\begin{align*}
\widehat{p}(z)^{m_1 + m_2} &= \Plus_{j \in \mathbb{N}} p_j^{m_1 + m_2} (z^{m_1+m_2})^j = \Plus_{j \in \mathbb{N}} (p_j x^j)^{m_1} (p_j y^j)^{m_2} \\
&\leqslant (\Plus_{j \in \mathbb{N}} (p_j x^j)^{m_1}) (\Plus_{j \in \mathbb{N}} (p_j y^j)^{m_2}) \\
&\leqslant (\Plus_{j \in \mathbb{N}} p_j x^j)^{m_1} (\Plus_{j \in \mathbb{N}} p_j y^j)^{m_2} = \widehat{p}(x)^{m_1} \widehat{p}(y)^{m_2}, 
\end{align*}
as required.
\end{proof}

\begin{proposition}\label{prop:infcom}
Let $p \in \cis[X]$. 
Then 
\[
\widehat{p}(\bigvee A) = \bigvee \widehat{p}(A), 
\]
for all nonempty subsets $A$ with sup, and 
\[
\widehat{p}(\bigwedge A) = \bigwedge \widehat{p}(A), 
\]
for all nonempty subsets $A$ with non-zero inf. 
If $p_0 = 0$ or $A$ is finite, then the condition $\bigwedge A \neq 0$ can be removed. 
In particular, $\widehat{p}$ is always a lattice morphism, in the sense that 
\begin{align*}
\widehat{p}(x \plus y) &= \widehat{p}(x) \plus \widehat{p}(y), \\
\widehat{p}(x \wedge y) &= \widehat{p}(x) \wedge \widehat{p}(y), 
\end{align*}
for all $x, y \in \cis$. 
\end{proposition}

\begin{proof}
The first part of the proposition comes from a combination of Equations~\eqref{eq:infsupf1}, \eqref{eq:infsupf}, \eqref{eq:infplus1}, \eqref{eq:infplus}, and Lemma~\ref{prop:sn}. 
If $\bigwedge A = 0$ with $A$ finite, then $0 \in A$, so $\widehat{p}(\bigwedge A) = p_0 = \bigwedge \widehat{p}(A)$. 

Now suppose that $p_0 = 0$ and $\bigwedge A = 0$, with $A$ not necessarily finite. 
We need to prove that $\bigwedge \widehat{p}(A) = 0$. 
By hypothesis $p_0 = 0$, so we can write $p(X) = X q(X)$ for some $q \in \cis[X]$. 
Let $\ell \in \cis$ be a lower bound of $\{ \widehat{p}(a) : a \in A \}$, and let $a_0 \in A$. 
Take $b := a \wedge a_0$. 
Then $\ell \leqslant \widehat{p}(a) \wedge \widehat{p}(a_0) = \widehat{p}(b) = b \widehat{q}(b) \leqslant a \widehat{q}(a_0)$. 
This implies that $\ell \leqslant (\bigwedge A) \widehat{q}(a_0) = 0$, so $\ell = 0$. 
This proves that the inf of $\{ \widehat{p}(a) : a \in A \}$ is $0$. 
\end{proof}

\begin{remark}
Proposition~\ref{prop:infcom} shows in particular that every polynomial function is Scott-continuous, and that every polynomial function $\widehat{p}$ such that $\widehat{p}(0) = 0$ is Lawson-continuous and dually Lawson-continuous (see Gierz et al.\ \cite{Gierz03} for the definitions). 
\end{remark}


\begin{proposition}[Golan]\label{prop:golan}
Let $p \in \cis[X]$. 
If $p_0 = 0$ and $\deg p > 0$, then $\widehat{p}$ is an order-embedding, in the sense that
\[
{ \widehat{p}(x) \leqslant \widehat{p}(y) } \Leftrightarrow { x \leqslant y },
\]
for all $x, y \in \cis$. 
In particular, $\widehat{p}$ is injective. 
\end{proposition}

\begin{proof}
See \cite[Proposition~2.22]{Golan03}. 
\end{proof}

\section{Closable and closed polynomials}\label{sec:clo}

This section first introduces the notion of a \textit{closable} polynomial. 
We prove that a polynomial $p(X)$ is closable if and only if there is a greatest polynomial $q(X)$ such that $p(X)$ and $q(X)$ have the same polynomial function. 
Denoting this polynomial by $\overline{p}(X)$, it is then natural to define a polynomial $p(X)$ as  \textit{closed} if it is closable and such that $p(X) = \overline{p}(X)$. 
Closable and closed polynomials will reveal their importance in the next section with our ``fundamental theorems'', where we will notably show that a polynomial is closed if and only if it can be decomposed as a product of polynomials of degree $1$. 

\subsection{Closable polynomials}

A polynomial $p \in \cis[X]$ is called \textit{closable} if the subset $\{ \widehat{p}(x) x^{-j} : x \in \cis^* \}$ has an inf for every $j \in \mathbb{N}$. 
Whenever this inf exists, we denote it by $\overline{p}_j$, i.e.\
\begin{equation}\label{eq:pjbar}
\overline{p}_j := \bigwedge_{x \in \cis^*} \widehat{p}(x) x^{-j}. 
\end{equation}
For a closable polynomial $p(X)$, the map $j \mapsto \overline{p}_j$ is reminiscent of the Legendre--Fenchel transform of $\widehat{p}$ in the sense of convex analysis, though turned here into a multiplicative rather than an additive form. 

The following lemma will be used repeatedly in the remaining part of this paper. 

\begin{lemma}\label{lem:0n}
Let $p \in \cis[X]$, $v = \val p$, and $n = \deg p$. 
Then we have: 
\begin{enumerate}
  \item\label{lem:0n0} $\overline{p}_0 = p_0^{}$,
  \item\label{lem:0n1} $\overline{p}_v = p_v^{}$, 
  \item\label{lem:0n2} $\overline{p}_n = p_n^{}$, 
  \item\label{lem:0n3} $\overline{p}_j = 0$, for all $j > n$ and all $j < v$. 
\end{enumerate}
\end{lemma}

\begin{proof}
Take 
\[
A(p, j) := \{ \widehat{p}(x) x^{-j} : x \in \cis^* \},
\]
for all $j \in \mathbb{N}$. 

\eqref{lem:0n1}. 
We prove that $\bigwedge A(p, v) = p_v$. 
We clearly have $\widehat{p}(x) x^{-v} \geqslant p_v$ for all $x \in \cis^*$, i.e.\ $p_v$ is a lower bound of $A(p, v)$. 
Now let $\ell \in \cis$ be another lower bound of $A(p, v)$, and let us show that $p_v \geqslant \ell$. 
We consider the polynomial $q(X) = p^{(v)}(X)$. 
For every $j \in \supp q$, we take $y_j^{} := (q_j^{-1} q_0^{}) \wedge 1$. 
Then we have 
\[
q_j^{} y_j^j \leqslant q_j^{} y_j^{} = q_0^{} \wedge q_j^{} \leqslant q_0^{},
\]
for all $j \in \supp q$. 
Take $y := \bigwedge_{j \in \supp q} y_j$. 
Then 
\[
\widehat{q}(y) = \Plus_{j \in \supp q} q_j^{} y^j \leqslant \Plus_{j \in \supp q} q_j^{} y_j^j \leqslant q_0^{}. 
\] 
Note that $y \neq 0$ since $q_0 = p_v \neq 0$. 
Thus, $p_v = q_0 \geqslant \widehat{q}(y) = \widehat{p}(y) y^{-v} \geqslant \ell$. 
This proves that $p_v$ is the greatest lower bound of $A(p, v)$, i.e.\ its inf. 

\eqref{lem:0n2}. 
We now prove that $\bigwedge A(p, n) = p_n$. 
We see that $p_n$ is a lower bound of $A(p, n)$. 
Let $m$ be another lower bound of $A(p, n)$. 
Take $y := p_n^{-1}(p_0^{} \plus \cdots \plus p_n^{})$. 
Then $p_n y^n \geqslant \widehat{p}(y)$, so that $p_n \geqslant \widehat{p}(y) y^{-n} \geqslant m$. 
This proves that $p_n$ is the greatest lower bound of $A(p, n)$, i.e.\ its inf. 

\eqref{lem:0n3}. 
Let us show that $\bigwedge A(p, j) = 0$, for all $j > n$ (the case $j < v$ is similar). 
Let $\ell \in \cis$ be a lower bound of $A(p, j)$, and let $u \in \cis^*$. 
Take $x_0 := 1 + u^{-1} \Plus_{k \in \supp p} p_k$. 
Then $x_0 \geqslant 1$, and $u x_0^{j-k} \geqslant u x_0^{} \geqslant p_k$, for all $k \in \supp p$. 
Hence, $p_k^{} x_0^k \leqslant u x_0^{j}$, for all $k \in \supp p$. 
This shows that $\widehat{p}(x_0^{}) x_0^{-j} \leqslant u$. 
This entails $\ell \leqslant u$, for all $u \in \cis^*$. 
By Lemma~\ref{lem:inf0bis}, we obtain $\ell = 0$, as required. 

\eqref{lem:0n0}. 
If $p_0 = 0$, then $v > 0$, so $\overline{p}_0 = 0 = p_0^{}$ by \eqref{lem:0n3}. 
If $p_0 \neq 0$, then $v = 0$, so $\overline{p}_0 = p_0^{}$ by \eqref{lem:0n1}. 
\end{proof}


Since $\overline{p}_j = 0$ whenever $j > \deg p$, the sequence of $\overline{p}_j$'s defines a polynomial $\overline{p} \in \cis[X]$, which we call the \textit{closure} of the closable polynomial $p$. 

\begin{lemma}\label{lem:pbar}
For every closable polynomial $p \in \cis[X]$, we have 
\begin{enumerate}
  \item\label{lem:pbar1} $\deg p = \deg \overline{p}$;
  \item\label{lem:pbar2} $p(X) \leqslant \overline{p}(X)$;
  \item\label{lem:pbar3} $\widehat{p}(x) = \widehat{\overline{p}}(x)$, for all $x \in \cis$; 
  \item\label{lem:pbar4} $\overline{p}(X)$ is the greatest element of $\{ q \in \cis[X] : \widehat{p} = \widehat{q} \}$. 
\end{enumerate}
\end{lemma}

\begin{proof}
\eqref{lem:pbar1} is already clear from Lemma~\ref{lem:0n}. 

\eqref{lem:pbar2}. 
We have $p_j \leqslant \widehat{p}(x) x^{-j}$ for all $j \in \mathbb{N}$ and $x \in \cis^*$, so $p_j \leqslant \overline{p}_j$. This implies $p(X) \leqslant \overline{p}(X)$. 

\eqref{lem:pbar3}. 
From $p(X) \leqslant \overline{p}(X)$ we get $\widehat{p}(x) \leqslant \widehat{\overline{p}}(x)$, for all $x \in \cis$. 
For the reverse inequality, pick some $x \in \cis^*$. 
For every $j \in \mathbb{N}$, the definition of $\overline{p}_j$ implies that $\overline{p}_j x^j \leqslant \widehat{p}(x)$, hence $\widehat{\overline{p}}(x) = \Plus_{j \geqslant 0} \overline{p}_j x^j \leqslant \widehat{p}(x)$. 
Moreover, $\widehat{p}(0) = p_0^{} = \overline{p}_0 =  \widehat{\overline{p}}(0)$ by Lemma~\ref{lem:0n}\eqref{lem:0n0}, so that $\widehat{p}(x) = \widehat{\overline{p}}(x)$, for all $x \in \cis$. 

\eqref{lem:pbar4}. 
Let $q \in \cis[X]$ with $\widehat{p} = \widehat{q}$. 
Then $q(X)$ is closable and $\overline{p}_j = \overline{q}_j$ for all $j \in \mathbb{N}$, i.e.\ $\overline{p}(X) = \overline{q}(X) \geqslant q(X)$. 
This shows that $\overline{p}(X)$ is the greatest element of $\{ q \in \cis[X] : \widehat{p} = \widehat{q} \}$. 
\end{proof}

\subsection{Closed polynomials}

We say that a polynomial $p \in \cis[X]$ is \textit{closed} if it is closable and $p = \overline{p}$, i.e.\  
\begin{equation}\label{eq:pj}
p_j = \bigwedge_{x \in \cis^*} \widehat{p}(x) x^{-j}, 
\end{equation}
for all $j \in \mathbb{N}$. 
By Lemma~\ref{lem:0n}, Equation~\eqref{eq:pj} already holds for indices $j$ such that $j \leqslant \val p$ or $j \geqslant \deg p$. 

\begin{remark}
A closable polynomial was called a \textit{closed} polynomial by Baccelli et al. \cite[Definition~3.87]{Baccelli92}. 
A closed polynomial was called a \textit{maximally represented polynomial} by Tsai \cite{Tsai12} and a \textit{Newton polynomial} by Rump \cite{Rump15}; the latter term relates to the \textit{Newton polygon} associated with a polynomial when $\cis = \mathbb{R}_{\max}$, see e.g.\ Akian et al. \cite[Proposition~2.7]{Akian25}. 
\end{remark}

\begin{lemma}
Let $p \in \cis[X]$. 
If $p(X)$ is closable, then $\overline{p}(X)$ is closed, and is the least closed polynomial greater than $p(X)$. 
\end{lemma}

\begin{proof}
Since $\widehat{p}$ and $\widehat{\overline{p}}$ coincide on $\cis$ by Lemma~\ref{lem:pbar}, we have 
\[
\bigwedge_{x \in \cis^*} \widehat{p}(x) x^{-j} = \bigwedge_{x \in \cis^*} \widehat{\overline{p}}(x) x^{-j},
\] 
for all $j \in \mathbb{N}$. 
This implies that $\overline{p}(X)$ is a closed polynomial. 

Now let $q \in \cis[X]$ be a closed polynomial such that $p(X) \leqslant q(X)$. 
This is a direct consequence of the definitions that $\overline{p}(X) \leqslant \overline{q}(X) = q(X)$. 
Thus, $\overline{p}(X)$ is indeed the least closed polynomial greater than $p(X)$. 
\end{proof}


Note that every polynomial of degree $\leqslant 1$ is closed. Theorem~\ref{thm:ft1} below will give more. 
A polynomial $p$ has \textit{full support} if 
\[
\supp p = \{ \val p, \val p + 1, \ldots, \deg p \},
\]
or equivalently if $p_{j} = 0$ implies $p_{j-1} = 0$, for all $0 < j < \deg p$.    

\begin{lemma}\label{lem:support}
Let $p \in \cis[X]$ be a polynomial of degree $> 0$. 
If $p$ is closed, then $p$ has full support. 
\end{lemma}

\begin{proof}
%
Let $n = \deg p$. 
To prove that $p$ has full support, we suppose that $p_j \neq 0$ for some $0 \leqslant j < n-1$, and we show that $p_{j+1} \neq 0$. 
For convenience we let $k = n-j$. 
Take $a := p_n^{-1}p_j^{}$. 
Since $a \neq 0$, we can define 
\[
b := (a^{k-1} \plus a^{1-k})^{-1}.
\] 
Then $b^k \leqslant a^{k-1}$. 
Let $x \in \cis$. 
Then $b^k x^k \leqslant a^{k-1} x^k$. 
Moreover, the binomial identity of Lemma~\ref{lem:bi} implies $a^{k-1} x^k \leqslant (x^k \plus a)^k$. 
Thus, $b^k x^k \leqslant (x^k \plus a)^k$, so $b x \leqslant x^k \plus a$ by Lemma~\ref{prop:perfore}. This shows that $p_n b x^{j+1} \leqslant p_n x^n \plus p_j x^j \leqslant \widehat{p}(x)$, for all $x \in \cis$. 
Consequently, 
\[
p_n b \leqslant \bigwedge_{x \in \cis^*} \widehat{p}(x) x^{-(j+1)} = p_{j+1}.
\]
Since $p_n b \neq 0$, we obtain $p_{j+1} \neq 0$. 
\end{proof}

\section{First and second fundamental theorems}\label{sec:ft}

In this section, we prove that a polynomial $p$ is closed (resp.\ closable) if and only if $p$ (resp.\ $\widehat{p}$) splits. 
Calling an idempotent semifield \textit{algebraically closed} if every polynomial function splits, we thus obtain the characterization that an idempotent semifield is algebraically closed if and only if every polynomial is closable. 
It is then a trivial consequence that every complete idempotent semifield, in particular $\mathbb{R}_{\max}$, is algebraically closed. 
In passing, we prove that a basket of roots of size $k < n$ of a closed polynomial of degree $n$ can be completed to a basket of roots of size $n$. 

\subsection{First fundamental theorem}

Let $p \in \cis[X]$ of degree $> 0$. 
If we can write 
\[
p(X) = (X \plus r_1) \cdots (X \plus r_k) q(X),
\]
for some $r_1, \ldots, r_k \in \cis$ and $q \in \cis[X]$, we call the tuple $(r_1, \ldots, r_k)$ a \textit{basket of roots} of $p(X)$. 
If $k = \deg p$, we say that the basket of roots is \textit{full}. 
This amounts to saying that $p(X)$ \textit{splits} in $\cis$, in the sense that 
\begin{equation}\label{eq:sp}
p(X) = p_n (X \plus r_1) \cdots (X \plus r_n), 
\end{equation}
for some $r_1, \ldots, r_n \in \cis$, with $n = \deg p$. 

Uniqueness---up to permutations---of $(r_1, \ldots, r_n)$ in Equation~\eqref{eq:sp} is not guaranteed in general. 
For instance, for a polynomial of degree $2$ it is always true that  
\[
(X \plus r_1) (X \plus r_2) = (X \plus c_1) (X \plus c_2),
\]
with $c_1 = r_1 \vee r_2$ and $c_2 = r_1 \wedge r_2$, thanks to the modular identity given by Equation~\eqref{eq:modular}, while in general $\{ r_1, r_2 \} \neq \{ c_1, c_2 \}$ (unless $\cis$ is totally ordered). 
See also Example~\ref{ex:qdiv2}. 

\begin{lemma}\label{lem:partial}
Let $p \in \cis[X]$ be a polynomial of degree $n > 0$. 
If 
\[
p(X) = p_n (X \plus c_1) \cdots (X \plus c_n)
\]
with $c_1 \geqslant \ldots \geqslant c_n$, then 
\[
p'(X) = p_n (X \plus c_1) \cdots (X \plus c_{n-1}).
\]
\end{lemma}

\begin{proof}
Using the formula for deriving a product of polynomials, we get 
\[
p'(X) = p_n \Plus_{j=1}^n \prod_{i \neq j} (X \plus c_i). 
\]
Now, for all $j \in \mathbb{N}$, we have 
\[
\prod_{i \neq j} (X \plus c_i) \leqslant \prod_{i \neq n} (X \plus c_i), 
\]
so that $p'(X) = p_n (X \plus c_1) \cdots (X \plus c_{n-1})$. 
\end{proof}

\begin{lemma}[Existence and uniqueness of corners]\label{lem:split}
Let $p \in \cis[X]$ be a polynomial of degree $n > 0$. 
If $p$ splits in $\cis$, then there exists a unique full basket of roots $(c_1, \ldots, c_n)$ of $p(X)$ with $c_1 \geqslant \ldots \geqslant c_n$. 
\end{lemma}

We call $(c_1, \ldots, c_n)$ the \textit{basket of corners} of the polynomial $p(X)$. 

\begin{proof}
(1): Existence. 
We prove by induction on $n > 0$ the property $P_n$: ``If $p(X) = p_n (X \plus r_1) \cdots (X \plus r_n)$, then $p(X) = p_n (X \plus c_1) \cdots (X \plus c_n)$ with $c_1 = r_1 \plus \cdots \plus r_n$ and $c_1 \geqslant \ldots \geqslant c_n$''.
 
If $n = 1$ this is clear. 
Suppose that $P_n$ holds for some $n > 0$, and let $p \in \cis[X]$ such that $p(X) = p_{n+1} (X \plus r_1) q(X)$, where 
\[
q(X) := (X \plus r_2) \cdots (X \plus r_{n+1}). 
\]
Applying the property $P_n$ to $q(X)$ we can write 
\[
q(X) = (X \plus s_2) \cdots (X \plus s_{n+1}),
\]
with $s_2 := r_2 \plus \cdots \plus r_{n+1}$ and $s_2 \geqslant \ldots \geqslant s_{n+1}$. 
Since $(X \plus r_1) (X \plus s_2) = (X \plus c_1) (X \plus (r_1 \wedge s_2))$ where $c_1 := r_1 \plus s_2 = r_1 \plus \cdots \plus r_{n+1}$, we obtain $p(X) = p_{n+1} (X \plus c_1) b(X)$ with 
\[
b(X) := (X \plus (r_1 \wedge s_2)) (X \plus s_3) \cdots (X \plus s_{n+1}). 
\] 
Now applying the property $P_n$ to $b(X)$ gives $b(X) = (X \plus c_2) \cdots (X \plus c_{n+1})$ with $c_2 := (r_1 \wedge s_2) \plus s_3 \plus \cdots \plus s_{n+1} = (r_1 \wedge s_2) \plus s_3$ and $c_2 \geqslant \ldots \geqslant c_{n+1}$. 
Moreover, $c_2 \leqslant s_2 \leqslant c_1$. 
This shows that 
\[
p(X) = p_{n+1} (X \plus c_1) \cdots (X \plus c_{n+1})
\]
with $c_1 := r_1 \plus \cdots \plus r_{n+1}$ and $c_1 \geqslant \ldots \geqslant c_{n+1}$. 
Therefore, $P_{n+1}$ is proved. 
By induction, the lemma is proved. 

(2): Uniqueness. 
We proceed again by induction on $n > 0$. 
If $n = 1$ this is clear. 
Suppose that the property of uniqueness holds for all polynomial of degree $\leqslant n - 1$, with $n > 1$. 
Let $p(X)$ be a polynomial of degree $n$, and assume that 
\[
p(X) = p_n^{} (X \plus c_1^{}) \cdots (X \plus c_n^{}) = p_n^{} (X \plus c_1') \cdots (X \plus c_n'),
\]
with $c_1^{} \geqslant \ldots \geqslant c_n^{}$ and $c_1' \geqslant \ldots \geqslant c_n'$. 
If $\widehat{p}(0) = 0$, we can divide $p(X)$ by $X^{\val p}$ and apply the induction hypothesis. 
Now suppose that $\widehat{p}(0) \neq 0$. 
By Lemma~\ref{lem:partial}, we have 
\[
p'(X) = p_n^{} (X \plus c_1^{}) \cdots (X \plus c_{n-1}^{}) = p_n^{} (X \plus c_1') \cdots (X \plus c_{n-1}').
\] 
From the induction hypothesis we obtain $c_k^{} = c_k'$, for all $1 \leqslant k < n$. 
Moreover, $\widehat{p}(0) = p_n^{} c_1^{} \cdots c_n^{} = p_n^{} c_1' \cdots c_n'$. 
Since $\widehat{p}(0) \neq 0$, we deduce that $c_n^{} = c_n'$. 
It follows that $c_k^{} = c_k'$, for all $1 \leqslant k \leqslant n$, and the result is proved. 
\end{proof}

\begin{figure}
  \begin{center}
    \begin{tikzpicture}
    \node (top) at (3,6) {$(\alpha^2, \alpha \beta, \beta^2)$};
    \node (a1) at (0,4) {$(\alpha, \alpha^2 \beta, \beta^2)$};
    \node (b1) at (3,4) {$(\alpha^2, \beta, \alpha \beta^2)$};
    \node (a2) at (0,2) {$(1, \alpha^2 \beta, \alpha \beta^2)$};
    \node (b2) at (3,2) {$(\alpha, \beta, \alpha^2 \beta^2)$};
    \node (bot) at (6,0) {$(1, \alpha \beta, \alpha^2 \beta^2)$};	
    \draw[->] (top) to[bend right] (a1);
    \draw[->] (top) to (b1);
    \draw[->] (a1) to (a2);
    \draw[->] (a1) to (b2);
    \draw[->] (b1) to (a2);
    \draw[->] (b1) to (b2);
    \draw[->] (a2) to[bend right] (bot);
    \draw[->] (b2) to (bot);
    \draw[->] (top) to[bend left] (bot);
    \end{tikzpicture}
  \end{center}
  \caption{Directed graph of the full baskets of roots (up to permutations) of the polynomial $p(X)$, 
  see Example~\ref{ex:qdiv2}. }
  \label{fig:v}
\end{figure}
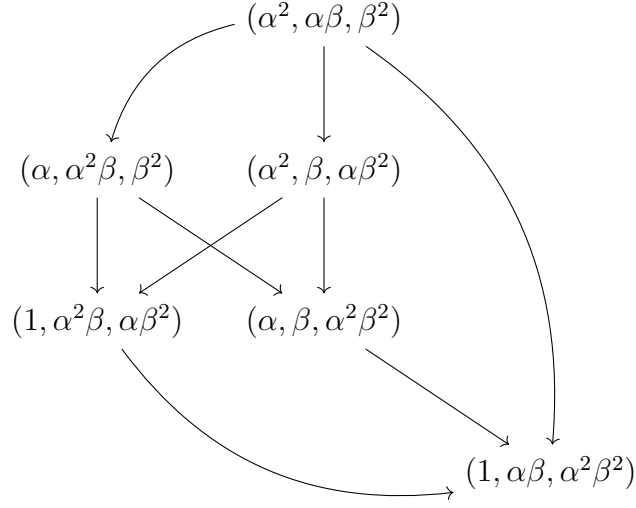


\begin{example}[Example~\ref{ex:qdiv} continued]\label{ex:qdiv2}
Let $\alpha, \beta$ be distinct prime numbers in $\mathbb{N}$, and consider the polynomial $p \in \mathbb{Q}_+^{\divi}[X]$ defined by
\[
p(X) = X^3 \vee \alpha^2 \beta^2 X^2 \vee \alpha^3 \beta^3 X \vee \alpha^3 \beta^3. 
\]
Each node of the directed graph of Figure~\ref{fig:v} gives a full basket of roots of $p(X)$. 
We draw an arc between two nodes if we can replace two elements $a, b$ of the source node by $a \vee b, a \wedge b$ in the target node. 
The sink of the directed graph, $(1, \alpha \beta, \alpha^2 \beta^2)$, is the basket of corners of $p(X)$. 
\end{example}

We have seen with Proposition~\ref{prop:convex} that every polynomial function satisfies a certain convexity property. 
We now say that a polynomial $p$ \textit{has the concavity property} if it satisfies 
\[
p_j^{m_1+m_2} \geqslant p_{j_1}^{m_1} p_{j_2}^{m_2}, 
\]
for all $j, j_1, j_2, m_1, m_2 \in \mathbb{N}$ such that $(m_1+m_2) j = m_1 j_1 + m_2 j_2$. 

The following theorem originates from Cuninghame-Green and Meijer \cite[Theorem~8]{Cuninghame80}, who discovered the fundamental result of factorization of polynomials with coefficients in the tropical semifield $\mathbb{R}_{\max}$. 
It notably extends Baccelli et al.\ (\cite[Lemma~3.41]{Baccelli92}, \cite[Theorem~3.43]{Baccelli92}) and Rump \cite[Proposition~3]{Rump15}. 

\begin{theorem}[First fundamental theorem]\label{thm:ft1}
Let $p \in \cis[X]$ be a polynomial of degree $n > 0$. 
Then the following conditions are equivalent: 
\begin{enumerate}
	\item\label{ft11} $p$ is closed; 
	\item\label{ft12} $p$ splits in $\cis$; 
	\item\label{ft13} $p$ has the concavity property;
	\item\label{ft14} $p$ has full support and $c_1 \geqslant \ldots \geqslant c_n$, 
\end{enumerate}
where $c_k := p_{n-k}^{} \cdot p_{n-k+1}^{-1}$ if $p_{n-k+1} \neq 0$, and $c_k := 0$ otherwise. 
If these conditions are satisfied, then 
$
p(X) = p_n (X \plus c_1) \cdots (X \plus c_n)
$, 
and the inf of $\{ \widehat{p}(x) x^{-j} : x \in \cis^* \}$ is reached at $c_{n-j}$, i.e.\ 
\[
p_j^{j+1} = \widehat{p}(c_{n-j}^{}) p_{j+1}^j, 
\]
for every $0 \leqslant j < n$. 
\end{theorem}

\begin{proof}
\eqref{ft14} $\Rightarrow$ \eqref{ft12}. 
We prove this implication by induction on $n$. 
If $n = 1$ this is clear. 
Suppose that it holds for $n-1$ with $n > 1$. 
Take $q(X) := p_n (X \plus c_1) \cdots (X \plus c_n)$. 
Then $q'(X) = p_n (X \plus c_1) \cdots (X \plus c_{n-1})$ by Lemma~\ref{lem:partial}. 
Since $p'(X)$ has full support, the induction property applies and we obtain $p'(X) = q'(X)$. 
Moreover, since $p(X)$ has full support, we have $p_0 = p_n c_1 \cdots c_n = q_0$.  
Thus, $p(X) = X p'(X) \plus p_0 = X q'(X) \plus p_0 = X q'(X) \plus q_0 = q(X)$. 
Therefore, the induction property holds for $n$. 


\eqref{ft12} $\Rightarrow$ \eqref{ft11}. 
By Lemma~\ref{lem:split} we can write $p(X) = p_n (X \plus c_1) \cdots (X \plus c_n)$, with $c_1 \geqslant \ldots \geqslant c_n$, so that 
\[
\widehat{p}(x) = p_n (x \plus c_1) \cdots (x \plus c_n),  
\]
for all $x \in \cis$. 
Now if $0 \leqslant j < n$, we show that $p_j$ is the inf of $A(p, j) := \{ \widehat{p}(x) x^{-j} : x \in \cis^* \}$. 
We already know that $p_j$ is a lower bound of $A(p,j)$. 
Moreover, $\widehat{p}(c_{n-j}^{}) c_{n-j}^{-j} = p_n^{} c_1^{} \cdots c_{n-j}^{} = p_j^{}$. 
Hence $p_j$ is the inf of $A(p,j)$, and the inf is reached at $c_{n-j}$. 
This shows that $p$ is closed. 


\eqref{ft11} $\Rightarrow$ \eqref{ft13}. 
Let $j, j_1, j_2, m_1, m_2 \in \mathbb{N}$ such that $(m_1+m_2) j = m_1 j_1 + m_2 j_2$. 
If $x \in \cis^*$, then 
\[
{ (\widehat{p}(x) x^{-j})^{m_1+m_2} } = { \widehat{p}(x)^{m_1} x^{-m_1 j_1} \widehat{p}(x)^{m_2} x^{-m_2 j_2} } \geqslant { p_{j_1}^{m_1} p_{j_2}^{m_2} }. 
\]
Thanks to Lemma~\ref{prop:sn} we deduce that 
\[
{ p_j^{m_1+m_2} } = { \left(\bigwedge_{x \in \cis^*} \widehat{p}(x) x^{-j}\right)^{m_1+m_2} } = { \bigwedge_{x \in \cis^*} (\widehat{p}(x) x^{-j})^{m_1+m_2} } \geqslant { p_{j_1}^{m_1} p_{j_2}^{m_2} }. 
\]
This proves that $p$ has the concavity property. 

\eqref{ft13} $\Rightarrow$ \eqref{ft14}. 
For all $0 < j \leqslant n$, the concavity property implies  
\[
p_j^{n-j+1} \geqslant p_{j-1}^{n-j} p_n^{}, 
\]
and
\[
p_j^2 \geqslant p_{j-1}^{} p_{j+1}^{}.   
\] 
The former inequality implies that $p$ has full support. 
Combined with the latter inequality this shows that ($p$ has full support and) $c_1 \geqslant \cdots \geqslant c_n$. 
\end{proof}

\begin{example}\label{ex:deg2}
Let $p(X) = a X^2 \plus b X \plus c$ be a polynomial of $\cis[X]$ of degree $2$ ($a \neq 0$). 
Then $p$ is closed if and only if $b^2 \geqslant a c$. 
In this case, $p(X) = a (X \plus a^{-1} b) (X \plus b^{-1} c)$ if $b \neq 0$, and $p(X) = a X^2$ otherwise. 
\end{example}

Given a semifield $\cis$, an \textit{extension} of $\cis$ consists of a semifield $K$ together with an injective morphism of semifields $\cis \rightarrow K$. 
We identify $\cis$ with its image in $K$ by the injective morphism. 
Note that $\cis$ is idempotent if and only if every of its extensions is idempotent. 

\begin{corollary}\label{coro:sp}
Let $p \in \cis[X]$ of degree $n > 0$. 
Then $p(X)$ splits (is closed) in some extension of $\cis$ if and only if $p(X)$ splits (is closed) in $\cis$. 
\end{corollary}

\begin{proof}
Let $K$ be an extension of $\cis$ such that $p(X)$ splits in $K$. 
By Theorem~\ref{thm:ft1}, we have $p(X) = p_n (X \plus c_1) \cdots (X \plus c_n)$, where $c_k := p_{n-k}^{} \cdot p_{n-k+1}^{-1}$ if $p_{n-k+1} \neq 0$, and $c_k = 0$ otherwise. 
We have $c_k \in \cis$ for all $k$, hence $p(X)$ splits in $\cis$. 
\end{proof}

\subsection{Completion of baskets of roots}

While the sum of polynomials does not preserve closability or closedness, various operations have better properties. 
We investigate this question in the following lines; this leads us to a result on completion of baskets of roots of a closed polynomial. 

\begin{lemma}\label{lem:prime}
The derivative of a closed polynomial is closed. 
\end{lemma}

\begin{proof}
Combine Lemma~\ref{lem:partial} with Theorem~\ref{thm:ft1}. 
\end{proof}

\begin{lemma}\label{lem:inf}
The inf of two closed polynomials is closed. 
\end{lemma}

\begin{proof}
Let $p, q \in \cis[X]$ be closed, take $r := p \wedge q$, and let $j, j_1, j_2, m_1, m_2 \in \mathbb{N}$ such that $(m_1+m_2) j = m_1 j_1 + m_2 j_2$. 
By Theorem~\ref{thm:ft1}, $p$ and $q$ have the concavity property, hence 
\begin{align*}
p_j^{m_1+m_2} &\geqslant p_{j_1}^{m_1} p_{j_2}^{m_2} \geqslant r_{j_1}^{m_1} r_{j_2}^{m_2},  \mbox{ and } \\
q_j^{m_1+m_2} &\geqslant q_{j_1}^{m_1} q_{j_2}^{m_2} \geqslant r_{j_1}^{m_1} r_{j_2}^{m_2}. 
\end{align*}
Thus, 
\[
r_j^{m_1+m_2} = p_j^{m_1+m_2} \wedge q_j^{m_1+m_2} \geqslant r_{j_1}^{m_1} r_{j_2}^{m_2}. 
\]
This proves that $r = p \wedge q$ has the concavity property, hence it is closed. 
%
\end{proof}

\begin{lemma}\label{lem:quotient}
The quotient of a closed polynomial by a non-zero polynomial is closed. 
\end{lemma}

\begin{proof}
Let $p, q \in \cis[X]$, with $p$ closed and $q \neq 0$. 
By Lemma~\ref{lem:prime}, the successive derivatives $p^{(1)}, p^{(2)}, \ldots$ are closed. 
Thus, by Lemma~\ref{lem:inf}, the inf 
$
\bigwedge_{i \in \supp q} q_i^{-1} p^{(i)}
$  
is closed, and this inf coincides with the quotient $p / q$ by Theorem~\ref{thm:crosby}. 
%
\end{proof}

\begin{lemma}\label{lem:product}
The product of two closed (resp.\ closable) polynomials is closed (resp.\ closable). 
Moreover, 
\[
\overline{p \cdot q} = \overline{p} \cdot \overline{q}, 
\]
for all closable polynomials $p, q \in \cis[X]$. 
\end{lemma}

\begin{proof}
Let $p, q \in \cis[X]$ be closed. 
By Theorem~\ref{thm:ft1}, $p$ and $q$ split in $\cis$, so their product splits in $\cis$, hence is closed. 

Now let $p, q \in \cis[X]$ be closable. 
If $p = 0$ or $q = 0$, then $p \cdot q = 0$ is closable, so we can suppose that $p \neq 0$ and $q \neq 0$. 
By the first part of the proof, $\overline{p} \cdot \overline{q}$ is closed; let us show that it is the least closed polynomial greater than $p \cdot q$. 
Let $r$ be a closed polynomial with $p \cdot q \leqslant r$. 
Then $p \leqslant r / q$, and since $r / q$ is closed by Lemma~\ref{lem:quotient}, we deduce that $\overline{p} \leqslant r / q$, hence $\overline{p} \cdot q \leqslant r$. 
Applying the same argument, we obtain $\overline{p} \cdot \overline{q} \leqslant r$. 
This proves that $p \cdot q$ is closable and $\overline{p \cdot q} = \overline{p} \cdot \overline{q}$. 
\end{proof}

\begin{proposition}[Completion of baskets of roots]
Let $p \in \cis[X]$ be a closed polynomial of degree $n > 0$. 
If $(r_1, \ldots, r_k)$, $k < n$, is a basket of roots of $p(X)$, then there exists some $r_{k+1}, \ldots, r_n \in \cis$ such that $(r_1, \ldots, r_n)$ is a full basket of roots of $p(X)$. 
\end{proposition}

\begin{proof}
Take $q(X) := (X \plus r_1) \cdots (X \plus r_k)$. 
Then the quotient $p / q$ is closed by Lemma~\ref{lem:quotient}. 
Moreover, $q(X)$ is a factor of $p(X)$, so by Corollary~\ref{coro:crosby} we have $p = q \cdot (p / q)$. 
It follows that $p / q$ is of degree $n - k$, and by Theorem~\ref{thm:ft1} we can write 
\[
(p / q)(X) = p_n (X \plus r_{k+1}) \cdots (X \plus r_n),
\]
for some $r_{k+1}, \ldots, r_n \in \cis$. 
Thus, $p(X) = p_n (X \plus r_1) \cdots (X \plus r_n)$, and $(r_1, \ldots, r_n)$ is a full basket of roots of $p(X)$. 
\end{proof}





\subsection{Algebraically closed idempotent semifields and second fundamental theorem}

Because of Corollary~\ref{coro:sp}, it is purposeless to look for an extension of a given idempotent semifield $\cis$ in which every $p \in \cis[X]$ splits. 
Yet it still makes sense to investigate extensions where factorization of polynomial functions always holds. 
We say that a polynomial function $\widehat{p}$ \textit{splits} in $\cis$ if it can be written as 
\[
\widehat{p}(x) = p_n (x \plus r_1) \cdots (x \plus r_n), 
\]
for all $x \in \cis$, where $n = \deg p > 0$ and $r_1, \ldots, r_n \in \cis$. 
In this case one can demand that $r_1 \geqslant \ldots \geqslant r_n$ (apply Lemma~\ref{lem:split} to the polynomial $p_n (X \plus r_1) \cdots (X \plus r_n)$). 

The following result extends Baccelli et al.\ \cite[Theorem~3.89]{Baccelli92}. 

\begin{theorem}[Second fundamental theorem]\label{thm:second}
Let $p \in \cis[X]$ be a polynomial of degree $n > 0$. 
Then the following conditions are equivalent:
\begin{enumerate}
  \item $p$ is closable;
  \item $\widehat{p}$ splits in $\cis$. 
\end{enumerate}
In this case, $\widehat{p}(x) = p_n (x \plus c_1) \cdots (x \plus c_n)$, for all $x \in \cis$, with $c_1 \geqslant \ldots \geqslant c_n$ the corners of $\overline{p}$, and the inf of $\{ \widehat{p}(x) x^{-j} : x \in \cis^* \}$ is reached at $c_{n-j}$, for all $0 \leqslant j < n$. 
\end{theorem}

\begin{proof}
First, assume that $p(X)$ is closable, and let us show that the polynomial function $\widehat{p}$ splits in $\cis$. 
Since $\overline{p}$ is a closed polynomial, we know by the first fundamental theorem that we can write 
\[
\overline{p}(X) = p_n (X \plus r_1) \cdots (X \plus r_n),  
\]
for some $r_1, \ldots, r_n \in \cis$, where $n = \deg p$. 
As a consequence,  
$
\widehat{\overline{p}}(x) = p_n (x \plus r_1) \cdots (x \plus r_n)
$, 
for all $x \in \cis$. 
By Lemma~\ref{lem:pbar}\eqref{lem:pbar3}, we have $\widehat{\overline{p}}(x) = \widehat{p}(x)$ for all $x \in \cis$, so that 
\[
\widehat{p}(x) = p_n (x \plus r_1) \cdots (x \plus r_n),  
\]
for all $x \in \cis$. 
This shows that the polynomial function $\widehat{p}$ splits in $\cis$. 

Conversely, assume that the polynomial function $\widehat{p}$ splits in $\cis$, and let $j \in \mathbb{N}$. 
We can assume that $\val p < j < \deg p$ by Lemma~\ref{lem:0n}. 
Since $\widehat{p}$ splits in $\cis$ we can write 
\[
\widehat{p}(x) = p_n (x \plus c_1) \cdots (x \plus c_n),  
\]
for all $x \in \cis$, with $c_1 \geqslant \ldots \geqslant c_n$, where $n = \deg p$. 
Now we take 
$
q_j = p_n c_1 \cdots c_{n-j}
$,
and we show that $q_j = \bigwedge_{x \in \cis^*} \widehat{p}(x) x^{-j}$, which will prove that $p(X)$ is closable. 
If $x \in \cis^*$ we have 
\[
c_1 \cdots c_{n-j} \leqslant (x \plus c_1) \cdots (x \plus c_{n-j}) (1 \plus c_{n-j+1} x^{-1}) \cdots (1 \plus c_n x^{-1}), 
\]
so that 
\[
q_j \leqslant p_n (x \plus c_1) \cdots (x \plus c_{n-j}) (x \plus c_{n-j+1}) \cdots (x \plus c_n) x^{-j}, 
\]
i.e.\ $q_j \leqslant \widehat{p}(x) x^{-j}$, for all $x \in \cis^*$. 
This shows that $q_j$ is a lower bound of $\{ \widehat{p}(x) x^{-j} : x \in \cis^* \}$. 
Moreover, we have $q_j^{} = \widehat{p}(c_{n-j}^{}) c_{n-j}^{-j}$ (note that $\val p < j < \deg p$ implies $c_{n-j} \neq 0$), so that $q_j = \bigwedge_{x \in \cis^*} \widehat{p}(x) x^{-j}$.  
\end{proof}

\begin{definition}
An idempotent semifield $\cis$ is \textit{algebraically closed} if, for every polynomial $p \in \cis[X]$ of degree $> 0$, the associated polynomial function $\widehat{p}$ splits in $\cis$. 
\end{definition}

As an immediate consequence of Theorem~\ref{thm:second} we get the following characterization of algebraically closed idempotent semifields. 

\begin{corollary}\label{coro:second}
An idempotent semifield is algebraically closed if and only if every polynomial is closable. 
\end{corollary}

An idempotent semifield is \textit{complete} if every subset bounded above has a sup (or equivalently if every nonempty subset has an inf). 

\begin{corollary}
Complete idempotent semifields are algebraically closed. 
\end{corollary}

\begin{corollary}\label{coro:zmax}
The following idempotent semifields: 
\begin{itemize}
	\item $\mathbb{Z}_{\max} = (\mathbb{Z} \cup \{-\infty\}, \max, -\infty, +, 0)$, 
	\item $\mathbb{R}_{\max} = (\mathbb{R} \cup \{-\infty\}, \max, -\infty, +, 0)$, 
	\item $\mathbb{R}_{+}^{\max} = (\mathbb{R}_+, \max, 0, \times, 1)$,  
\end{itemize}
are complete, hence algebraically closed. 
\end{corollary}

\begin{remark}
Cuninghame-Green and Meijer were the first to prove that $\mathbb{R}_{\max}$ is algebraically closed (see \cite[Theorem~~8]{Cuninghame80}). 
See also Baccelli et al.\ \cite[Theorem~3.43]{Baccelli92}, Akian, Bapat, and Gaubert \cite{Akian04}, \cite{Akian06}. 
\end{remark}

\begin{example}
We consider the algebraically closed idempotent semifield $(K_{\alpha}, \vee, 0, \times, 1)$, where 
\[
K_{\alpha} := \{ \alpha^{k} : k \in \mathbb{Z} \} \cup \{ 0 \}, 
\]
$\alpha$ is a real number $> 1$, and $\vee$ denotes the maximum operator, i.e.\ $x \vee y := \max(x, y)$. 
Note that $K_{\alpha}$ is isomorphic to $\mathbb{Z}_{\max}$. 
Take $p(X) := X^2 \vee \alpha X \vee \alpha$ and $q(X) := X^2 \vee \alpha$. 
Then $p(X)$ is closed by Example~\ref{ex:deg2}, while $q(X)$ is not, and 
$
p(X) = (X \vee 1) (X \vee \alpha)
$. 
Moreover,  
\[
\widehat{p}(x) = (x \vee 1) (x \vee \alpha) = (x \vee \alpha^{1/2})^2 = \widehat{q}(x), 
\]
for all $x \in K_{\alpha}$, where $\alpha^{1/2}$ belongs to the radicable closure $\widetilde{K}_{\alpha}$ of $K_{\alpha}$, see Theorem~\ref{thm:radclo} in the Appendix. 
Thus, $\widehat{q}$ splits in $K_{\alpha}$, hence $q(X)$ is closable by Theorem~\ref{thm:second}, and $\overline{q}(X) = p(X)$. 
Note that the closed polynomial associated with $q(X)$ in $\widetilde{K}_{\alpha}$ is not $p(X)$, but rather 
\[
(X \vee \alpha^{1/2})^2 = X^2 \vee \alpha^{1/2} X \vee \alpha. 
\]
\end{example}


%





\section{Preradicability and polynomial inequalities}\label{sec:ineq}

This section focuses on the concept of a \textit{preradicable} idempotent semifield, a property that is shown to be useful for solving polynomial inequalities. 
We also prove that every algebraically closed idempotent semifield is preradicable, and that the converse statement holds true if the idempotent semifield is totally ordered. 

\subsection{Definition and first properties}

An idempotent semifield $\cis$ is \textit{preradicable} if, for all $y \in \cis$, $n \in \mathbb{N}^*$, there exists a least element $x \in \cis$ such that $x^n \geqslant y$, equivalently if the maps $\theta_k : \cis \rightarrow \cis, x \mapsto x^k$ are residuated, for all $k \in \mathbb{N}^*$. 

\begin{lemma}\label{lem:pres}
Assume that $\cis$ is preradicable, 
and let $p \in \cis[X]$. 
Then the following conditions are equivalent: 
\begin{enumerate}
	\item\label{pres1} $\widehat{p}$ is residuated; 
	\item\label{pres2} $\widehat{p}$ is injective; 
	\item\label{pres4} $\widehat{p}$ is an order-embedding; 
	\item\label{pres3} $p_0 = 0$ and $\deg p > 0$. 
\end{enumerate}
In this case, $\widehat{p}^{\#}(y) = 0$ if and only if $y = 0$. 
\end{lemma}

\begin{proof}
For every $k \in \mathbb{N}^*$, we denote by $\theta_k$ the map $\cis \rightarrow \cis, x \mapsto x^k$. 


\eqref{pres4} $\Rightarrow$ \eqref{pres2} is evident. 

\eqref{pres2} $\Rightarrow$ \eqref{pres3}. 
If $\deg p \leqslant 0$, then $\widehat{p}$ is clearly not injective. 
Now assume that $p_0 \neq 0$. 
Take $x_0^{} := \bigwedge_{k \in \supp p} p_k^{-1} p_0^{}$, and let us show that $\widehat{p}(x_0) = \widehat{p}(0)$. 
Since $x_0^{} \leqslant p_0^{-1} p_0^{} = 1$, we have $x_0^k \leqslant x_0^{}$ for all $k \geqslant 1$, so $x_0^k \leqslant p_k^{-1} p_0^{}$ for all $k \in \supp p$. 
This implies $p_k^{} x_0^k \leqslant p_0^{}$, for all $k \in \supp p$, hence $\widehat{p}(x_0) \leqslant p_0$. 
It is obvious that $\widehat{p}(x_0) \geqslant p_0$, so that $\widehat{p}(x_0) = p_0 = \widehat{p}(0)$. 
Since $x_0 \neq 0$, this proves that $\widehat{p}$ is not injective. 

\eqref{pres3} $\Rightarrow$ \eqref{pres1}. 
Assume that $p_0 = 0$ and $\deg p > 0$. 
Since $\cis$ is preradicable, $\theta_k$ is residuated. 
If we define $\widehat{p}^{\#}$ by 
\begin{equation}\label{eq:psharp}
\widehat{p}^{\#}(y) = \bigwedge_{k \in \supp p} \theta_k^{\#}(p_k^{-1} y), 
\end{equation}
then $\widehat{p}(x) \leqslant y \Leftrightarrow x \leqslant \widehat{p}^{\#}(y)$ for all $x, y \in \cis$. 

\eqref{pres1} $\Rightarrow$ \eqref{pres3}. We have $0 \leqslant \widehat{p}^{\#}(0)$; this entails $p_0 = \widehat{p}(0) \leqslant 0$, i.e.\ $p_0 = 0$. 
Moreover, a constant polynomial function cannot be residuated with $\cis \neq \mathbb{B}$, hence $\deg p > 0$. 

\eqref{pres3} $\Rightarrow$ \eqref{pres4} is Proposition~\ref{prop:golan}. 

Now assume that $\widehat{p}$ is residuated and $y \neq 0$. 
Since $y \wedge 1 \leqslant 1$, we have $(y \wedge 1)^k \leqslant y \wedge 1 \leqslant y$, hence $y \wedge 1 \leqslant \theta_k^{\#}(y)$, for all $k \geqslant 1$. 
With Equation~\eqref{eq:psharp}, we obtain $\widehat{p}^{\#}(y) \geqslant \bigwedge_{k \in \supp p} (p_k^{-1} y \wedge 1) > 0$ (note that $k = 0 \notin \supp p$ since $p_0 = 0$). 
Conversely, if $y = 0$, then $\widehat{p}(\widehat{p}^{\#}(y)) \leqslant y = 0$, hence $\widehat{p}(\widehat{p}^{\#}(y)) = 0$, so $\widehat{p}^{\#}(y) = 0$. 
\end{proof}

\subsection{Preradicability and solutions to polynomial inequalities}

We now study polynomial inequalities of the form $\widehat{p}(x) \leqslant \widehat{q}(x)$. 

\begin{proposition}\label{prop:ineqc0}
Assume that $\cis$ is preradicable, 
and let $p, q \in \cis[X]$. 
If one of the following conditions holds:
\begin{enumerate}
  \item\label{prop:ineqc01} $\deg p < \deg q$, or
  \item\label{prop:ineqc02} $\val p > \val q$,
\end{enumerate}
then the inequality $\widehat{p}(x) \leqslant \widehat{q}(x)$ has a non-zero solution in $\cis$. 
\end{proposition}

\begin{proof}
Case \eqref{prop:ineqc01}. 
Since $\cis$ is preradicable, for every $k \in \mathbb{N}^*$ there is a map $\theta_k^{\flat}$ such that 
\[
x \leqslant y^k \Leftrightarrow \theta_k^{\flat}(x) \leqslant y, 
\]
for all $x, y \in \cis$. 
Note that $\theta_k^{\flat}(x) = 0 \Rightarrow x \leqslant 0^k \Rightarrow x = 0$. 
Take $n := \deg q$, $m := \deg p$, and $x_0^{} := \Plus_{j < n} \theta_{n - j}^{\flat}(q_n^{-1} p_j^{})$. 
Since $m < n$, we have $x_0 \geqslant \theta_{n - m}^{\flat}(q_n^{-1} p_m^{}) > 0$, so $x_0$ is non-zero. 
Moreover, $x_0^{n-j} \geqslant q_n^{-1} p_j^{}$, for all $j < n$. 
This yields $\widehat{q}(x_0^{}) \geqslant q_n^{} x_0^{n} \geqslant p_j^{} x_0^{j}$, for all $j < n$. 
Since $\deg p < n$, this implies $\widehat{q}(x_0) \geqslant \widehat{p}(x_0)$, which proves that $x_0$ is a non-zero solution. 

Case \eqref{prop:ineqc02}. 
Take $v := \val q$. 
Since $\val p > v$, the derivative $p^{(v)}$ of $p$ induces a residuated function by Lemma~\ref{lem:pres}. 
Take 
\[
x_0 := \widehat{p^{(v)}}^{\#}(q_v). 
\]
Then $x_0 \neq 0$ by the second part of Lemma~\ref{lem:pres}. 
Moreover, 
\[
\widehat{p}(x_0^{}) = x_0^{v} \widehat{p^{(v)}}(x_0^{}) \leqslant q_v^{} x_0^v \leqslant \widehat{q}(x_0^{}), 
\]
which proves that $x_0$ is a non-zero solution.  
\end{proof}

\begin{proposition}\label{prop:ineqc}
Assume that $\cis$ is preradicable, 
and let $p \in \cis[X]$ and $c \in \cis$. 
If $\deg p > 0$, then the inequality $\widehat{p}(x) \leqslant c$ has a (greatest) solution in $\cis$ if and only if $p_0 \leqslant c$. 
\end{proposition}

\begin{proof}
If $\widehat{p}(x) \leqslant c$ for some $x \in \cis$, then $p_0 \leqslant c$. 
Conversely, assume that $p_0 \leqslant c$. We write $p(X) = q(X) \plus p_0$, where $q(X)$ is the polynomial $X p'(X)$. By Lemma~\ref{lem:pres}, the map $\widehat{q}$ is residuated. 
Take $x_0 := \widehat{q}^{\#}(c)$. Then we have $\widehat{p}(x_0) = \widehat{q}(x_0) \plus p_0 \leqslant c \plus p_0 = c$, so we have proved the existence of a solution in $\cis$ to the inequality $\widehat{p}(x) \leqslant c$. 
If $x_1$ is another solution, then $\widehat{q}(x_1) \leqslant \widehat{p}(x_1) \leqslant c$, so $x_1 \leqslant \widehat{q}^{\#}(c) = x_0$, which proves that $x_0$ is the greatest solution. 
\end{proof}



\subsection{Links between preradicability and algebraic closedness}\label{subsec:prerad}

If $\cis$ is a complete idempotent semifield, it is preradicable as one can deduce from Lemma~\ref{prop:sn}. 
The following result strengthens this latter assertion. 

\begin{proposition}\label{prop:prerad}
Every algebraically closed idempotent semifield is preradicable. 
\end{proposition}

\begin{proof}
Let $\cis$ be an algebraically closed idempotent semifield, and let $t \in \cis$ and $n \in \mathbb{N}^*$. 
We can suppose that $t\neq 0$ without loss of generality. 
We consider the polynomial $p(X) = X^n \plus t$. 
Since $\cis$ is algebraically closed we can write 
\[
\widehat{p}(x) = (x \plus c_1) \cdots (x \plus c_n), 
\]
for all $x \in \cis$, with $c_1 \geqslant \ldots \geqslant c_n$. 
Moreover, 
\[
c_1^n \plus t = \widehat{p}(c_1)  = (c_1 \plus c_1) \cdots (c_1 \plus c_n) = c_1^n, 
\]
so $t \leqslant c_1^n$. 
To complete the proof, let us show that $c_1$ is the \textit{least} $s$ such that $t \leqslant s^n$. 
Let $s$ be some element of $\cis$ satisfying $t \leqslant s^n$, and let us show that $c_1 \leqslant s$. 
We have $s^n = s^n \plus t = \widehat{p}(s) = (s \plus c_1) \cdots (s \plus c_n)$. This implies that $1 = (1 \plus c_1 s^{-1}) \cdots (1 \plus c_n s^{-1})$, so each term of the product $(1 \plus c_1 s^{-1}) \cdots (1 \plus c_n s^{-1})$ equals $1$. In particular, $1 \plus c_1 s^{-1} = 1$, hence $c_1 s^{-1} \leqslant 1$, i.e.\ $c_1 \leqslant s$.  
%
\end{proof}

\begin{example}[Example~\ref{ex:deg2} continued]\label{ex:deg2bis}
If $\cis$ is preradicable, every polynomial of degree $2$ in $\cis[X]$ is closable. 
Indeed, let $p(X) = a X^2 \plus b X \plus c$ be a polynomial of $\cis[X]$ of degree $2$ ($a \neq 0$). 
Take $u := \theta_2^{\flat}(a c) \plus b$. 
Then $u^2 \geqslant a c$, so the polynomial $q(X) := a X^2 \plus u X \plus c$ is closed. 
Let $x \in \cis^*$. 
We have $(a x + b + c x^{-1})^2 \geqslant (a x + c x^{-1})^2 \geqslant a c$. 
Thus, $\theta_2^{\flat}(a c) \leqslant a x + b + c x^{-1}$. 
We also have $b \leqslant a x + b + c x^{-1}$. 
Hence, $u \leqslant a x + b + c x^{-1} = \widehat{p}(x) x^{-1}$, for all $x \in \cis^*$.  
Moreover, $u = \widehat{p}(x) x^{-1}$ with $x := a^{-1} u$. 
We get $u = \overline{p}_1$, so that $q(X) = \overline{p}(X)$. 
This proves that $p(X)$ is closable. 
\end{example}


We now prove a converse statement to the previous result in the case where the idempotent semifield $\cis$ is totally ordered. 

\begin{theorem}\label{thm:eqalg}
Assume that $\cis$ is totally ordered. 
Then the following conditions are equivalent: 
\begin{itemize}
	\item $\cis$ is algebraically closed; 
	\item $\cis$ is preradicable. 
\end{itemize}
\end{theorem}

\begin{proof}
If $\cis$ is algebraically closed, it is preradicable by Proposition~\ref{prop:prerad}. 
Conversely, assume that $\cis$ is preradicable. 
For every $k \in \mathbb{N}^*$, we denote by $\theta_k$ the map $\cis \rightarrow \cis, x \mapsto x^k$. 
Since $\cis$ is preradicable, there are maps $\theta_k^{\#}$ and $\theta_k^{\flat}$ such that 
\begin{align*}
x^k \leqslant y &\Leftrightarrow x \leqslant \theta_k^{\#}(y) \\
x \leqslant y^k &\Leftrightarrow \theta_k^{\flat}(x) \leqslant y, 
\end{align*}
for all $x, y \in \cis$. 
Now let $p \in \cis[X]$ be a polynomial of degree $n > 0$ and valuation $v$, and let $j \in \mathbb{N}$ with $v < j < n$. 
Take $A(p,j) := \{ \widehat{p}(x) x^{-j} : x\in \cis^* \}$, and let us show that $A(p, j)$ has an inf. 
For every $k \in \supp p$, take 
\[
m_k := \sum_{\ell < k, \ell \in \supp p} \theta_{k - \ell}^{\flat}(p_k^{-1} p_{\ell}^{})
\]
and 
\[
M_k := \bigwedge_{\ell > k, \ell \in \supp p} \theta_{\ell - k}^{\#}(p_{\ell}^{-1} p_k^{}). 
\]
Since $\cis$ is totally ordered, it is not difficult to show that $\cis^*$ can be decomposed as $\cis^* = \bigcup_k I_k$, 
where $I_v = (0, M_v]$, $I_k = [m_k, M_k]$ if $k \in \supp p$ and $v < k < n$, and $I_n = [m_n, \infty)$. 
Moreover, $x \in I_k$ if and only if $\widehat{p}(x) = p_k x^k$. 
Note that the interval $I_k$ can be empty, since we do not necessarily have $m_k \leqslant M_k$. 
Let $K$ denote the set of indices $k \in \supp p$ such that $I_k \neq \emptyset$. 
We deduce that 
\[
A(p,j) = \bigcup_{k \in K} \{ p_k x^{k-j} : x \in I_k \}. 
\]
We conclude that $A(p,j)$ has an inf given by  
\begin{equation}\label{eq:calctotord}
\overline{p}_j = \bigwedge_{k \in K, k > j} p_k^{} m_k^{k-j} \wedge \bigwedge_{k \in K, k \leqslant j} p_k^{} M_k^{k-j}. 
\end{equation}
This proves that $\cis$ is algebraically closed. 
\end{proof}

\section{Radicability and polynomial equations}\label{sec:eq}

In this section, we introduce \textit{radicability} as a property stronger than preradicability to qualify an idempotent semifield. 
This property attests to its strength as its provides existence of solutions to polynomial equations. 
We also prove that every radicable idempotent semifield is algebraically closed and order-dense, and that the converse statement holds true if the idempotent semifield is totally ordered. 
In the last part of the section, we consider \textit{rational polynomials}, defined as formal polynomials modulo a non-zero polynomial factor. 
We prove that, given a radicable idempotent semifield, there is an isomorphism of semirings between rational polynomials and polynomial functions, and that every rational polynomial splits. 

\subsection{Definition and first properties} 

An idempotent semifield $\cis$ is \textit{radicable} if, for all $y \in \cis$, $n \in \mathbb{N}^*$, there is an element (necessarily unique, by Lemma~\ref{prop:perfore}) $x \in \cis$ such that $x^n = y$, that we denote by $y^{1/n}$. 
Obviously, every radicable idempotent semifield is preradicable. 


The following lemma extends Shpiz \cite[Proposition~2]{Shpiz00}. 

\begin{lemma}\label{lem:pinv}
Assume that $\cis$ is radicable, 
and let $p \in \cis[X]$. 
Then the following conditions are equivalent: 
\begin{itemize}
	\item $\widehat{p}$ is residuated; 
	\item $\widehat{p}$ is surjective; 
	\item $\widehat{p}$ is injective; 
	\item $\widehat{p}$ is bijective; 
	\item $\widehat{p}$ is an order-embedding; 
	\item $p_0 = 0$ and $\deg p > 0$. 
\end{itemize}
\end{lemma}

\begin{proof}
Thanks to Lemma~\ref{lem:pres}, it suffices to show that $\widehat{p}$ is surjective if and only if $p_0 = 0$ and $\deg p > 0$. 
Assume that $\widehat{p}$ is surjective. 
Then there is some $x$ such that $\widehat{p}(x) = 0$. 
Since $0 \leqslant p_0 \leqslant \widehat{p}(x) = 0$, we get $p_0 = 0$. 
Moreover, as a $\cis$-valued surjective map, $\widehat{p}$ cannot be constant, hence $\deg p > 0$. 

Conversely, assume that $p_0 = 0$ and $\deg p > 0$. Let $y \in \cis$.  
By Lemma~\ref{lem:pres} $\widehat{p}$ is residuated, so we can take $x := \widehat{p}^{\#}(y)$. 
By residuation we know that $\widehat{p}(x) \leqslant y$, so we now prove that $\widehat{p}(x) \geqslant y$. 
Using Equation~\eqref{eq:psharp}, 
\[
x = \bigwedge_{k \in \supp p} (p_k^{-1} y)^{1/k}. 
\]
By Proposition~\ref{prop:infcom}, $\widehat{p}$ preserves arbitrary nonempty infs, so that 
\[
\widehat{p}(x) = \bigwedge_{k \in \supp p} \widehat{p}\left((p_k^{-1} y)^{1/k}\right)
\]
For every $k \in \supp p$, take $z_k := (p_k^{-1} y)^{1/k}$. 
Then $\widehat{p}(z_k) \geqslant p_k^{} z_k^k = p_k^{} p_k^{-1} y = y$. 
Thus, $\widehat{p}(x) \geqslant y$, so that $\widehat{p}(x) = y$. 
This proves that $\widehat{p}$ is surjective. 
\end{proof}

\subsection{Radicability and solutions to polynomial equations}

The following result is due to Shpiz \cite{Shpiz00} and can be seen as an intermediate value theorem for polynomial functions. 

\begin{theorem}[Shpiz]
Assume that $\cis$ is radicable, 
and let $p \in \cis[X]$ and $t \in \cis$. 
If $\deg p > 0$, then the equation $\widehat{p}(x) = t$ has a (greatest) solution in $\cis$ if and only if $p_0 \leqslant t$. 
Moreover, if $p_0 = 0$, then the solution is unique. 
\end{theorem}

\begin{proof}
The map $\widehat{q}$ defined in the proof of Proposition~\ref{prop:ineqc} is now invertible by Lemma \ref{lem:pinv}, and $x_0 = \widehat{q}^{\#}(t)$ is the greatest solution to the equation $\widehat{p}(x) = t$. 
If $p_0 = 0$, uniqueness is clear by Lemma~\ref{lem:pinv}. 
\end{proof}

Following the notion of a root of a polynomial proposed by Castella \cite[p.~7]{Castella13}, we say that $r \in \cis$ is a \textit{weak root} of the polynomial $p \in \cis[X]$  if $p(X) = a(X) \plus b(X)$ for some $a, b \in \cis[X]$ with disjoint supports such that $\widehat{a}(r) = \widehat{b}(r)$. 

\begin{corollary}
An idempotent semifield is radicable if and only if every polynomial of degree $> 0$ has a weak root. 
\end{corollary}

\begin{proof}
Let $\cis$ be radicable, and let $p \in \cis[X]$ of degree $> 0$. 
Take $q(X) := X p'(X)$. 
Then the (unique) solution of the equation $\widehat{q}(x) = p_0$ is a weak root of $p$. 

Conversely, assume that every polynomial of degree $> 0$ has a weak root. 
Consider in particular the polynomial $p(X) = X^n \plus t$, with $n \in \mathbb{N}^*$ and $t \in \cis$. 
Then a weak root $x$ of $p$ satisfies $x^n = t$. 
This shows that $\cis$ is radicable. 
\end{proof}

The following result was stated by Bogdanov \cite[Theorem~4]{Bogdanov04} without proof. 

\begin{theorem}[Bogdanov]\label{thm:bog}
Assume that $\cis$ is radicable, 
and let $p, q \in \cis[X]$. 
If $\val p > \deg q > -\infty$, then the equation $\widehat{p}(x) = \widehat{q}(x)$ has a least non-zero solution in $\cis$. 
\end{theorem}


\begin{proof}
Take $k := \deg q > -\infty$. 
Since $\val p > k$, the successive derivatives $p^{(0)}, \ldots, p^{(k)}$ of $p$ induce residuated functions by Lemma~\ref{lem:pres}. 
Take 
\[
x_0 := y_0 \plus \cdots \plus y_k \mbox{ , where } y_j = \widehat{p^{(j)}}^{\#}(q_j). 
\]
Then $x_0 \neq 0$ since $x_0 \geqslant y_k > 0$ by the second part of Lemma~\ref{lem:pres}. 
Moreover, $\widehat{p}(x_0) = \widehat{p}(y_0) \plus \cdots \plus \widehat{p}(y_k)$ by the additivity of $\widehat{p}$ (Proposition~\ref{prop:infcom}). 
It follows that $\widehat{p}(x_0^{}) \leqslant q_0^{} \plus \cdots \plus q_k^{} y_k^k \leqslant \widehat{q}(x_0^{})$. 
Moreover, for every $j \in \mathbb{N}$ with $0 \leqslant j \leqslant k$, we have $x_0 \geqslant y_j$, so 
\[
\widehat{p}(x_0^{}) = \widehat{p^{(j)}}(x_0^{})  x_0^{j} \geqslant q_j^{} x_0^{j}. 
\] 
This shows that $\widehat{p}(x_0) \geqslant \widehat{q}(x_0)$. 
Thus, we have found our non-zero $x_0$ such that $\widehat{p}(x_0) = \widehat{q}(x_0)$. 

To conclude the proof, we show that $x_0$ is the \textit{least} non-zero solution. 
Let $x$ be another non-zero solution. 
For every $j \in \mathbb{N}$ with $0 \leqslant j \leqslant k$, we have $\widehat{p}(x) = \widehat{q}(x) \geqslant q_j x^j$, hence 
\[
\widehat{p^{(j)}}(x) = \widehat{p}(x) x^{-j} \geqslant q_j. 
\]
By definition of $y_j$, we get $x \geqslant y_j$. 
This implies $x \geqslant y_0 \plus \cdots \plus y_k = x_0$, as required. 
\end{proof}

\begin{remark}[Fixed points]
Given a polynomial $p \in \cis[X]$, it is easily seen that the polynomial function $\widehat{p}$ has a fixed point (i.e.\ an element $x \in \cis$ with $\widehat{p}(x) = x$) if and only if $p_0 = 0$ or $\widehat{p}'(p_0) \leqslant 1$, in which case $p_0$ is the least fixed point of $\widehat{p}$. 
In particular, we may have no solution to the equation $\widehat{p}(x) = \widehat{q}(x)$ if $\val p \leqslant \deg q$: 
take for instance $q(X) = X$ and $p(X)$ such that $p_0 > 0$ and $\widehat{p}'(p_0) \not\leqslant 1$. 
\end{remark}

An idempotent semifield is \textit{$\sigma$-complete} if every countable subset bounded above has a sup. 

\begin{theorem}
Assume that $\cis$ is radicable and $\sigma$-complete,  
and let $p, q \in \cis[X]$. 
If $q_0 = 0$ and the inequality $\widehat{p}(x) \leqslant \widehat{q}(x)$ has a solution in $\cis$, then the equation $\widehat{p}(x) = \widehat{q}(x)$ has a solution in $\cis$. 
\end{theorem}

\begin{proof}
Let $y \in \cis$ be such that $\widehat{p}(y) \leqslant \widehat{q}(y)$. 
If $\deg q \leqslant 0$, then $\widehat{q}(y) = 0$, so $p_0 = 0$, and $x = 0$ is a solution of the equation $\widehat{p}(x) = \widehat{q}(x)$. 
Now assume that $\deg q > 0$. 
Since $\cis$ is radicable, $\widehat{q}$ is a bijection and an order-embedding by Lemma~\ref{lem:pinv}. 
Let $\widehat{q}^{\#}$ be the inverse map of $\widehat{q}$. 
Take $x_0 := 0$, and consider the sequence $(x_n)_{n \in \mathbb{N}}$ defined recursively by 
\[
x_{n+1} = \widehat{q}^{\#}(\widehat{p}(x_{n})),
\]
for all $n \in \mathbb{N}$. 
We prove by induction that $x_{n} \leqslant y$, for all $n \in \mathbb{N}$. 
This holds for $n = 0$.  
Moreover, if $x_{n} \leqslant y$ for some $n \in \mathbb{N}$, then $\widehat{q}(x_{n+1}) = \widehat{p}(x_{n}) \leqslant \widehat{p}(y) \leqslant \widehat{q}(y)$, so that $x_{n+1} \leqslant y$. 

Now we prove by induction that $x_{n} \leqslant x_{n+1}$, for all $n \in \mathbb{N}$. 
This holds for $n = 0$. 
Moreover, if $x_{n} \leqslant x_{n+1}$ for some $n \in \mathbb{N}$, then $\widehat{q}(x_{n+1}) = \widehat{p}(x_{n}) \leqslant \widehat{p}(x_{n+1}) = \widehat{q}(x_{n+2})$, so that $x_{n+1} \leqslant x_{n+2}$. 

Thus, $(x_n)_{n \in \mathbb{N}}$ is a non-decreasing sequence bounded above. 
Since $\cis$ is $\sigma$-complete, we can take $x := { \bigvee_{n \in \mathbb{N}} x_n }$. 
Using Proposition~\ref{prop:infcom} and the fact that $\widehat{q}(x_{n+1}) = \widehat{p}(x_{n})$ for all $n \in \mathbb{N}$, we obtain $\widehat{p}(x) = \widehat{q}(x)$, as required. 
\end{proof}



\begin{remark}
Rump \cite[Corollary, p.~42]{Rump16} showed that the equation 
\[
a x \plus b = c x \plus d,
\]
with $a, c \in \cis^*$ and $b, d \in \cis$, has a solution if and only if 
\begin{align*}
a d \wedge b c &\leqslant a b \leqslant a d \vee b c \quad \mbox{ and } \\
a d \wedge b c &\leqslant c d \leqslant a d \vee b c. 
\end{align*}
In this case, there is a unique solution $x$ such that $\alpha \wedge \beta \leqslant x \leqslant \alpha \vee \beta$, with $\alpha := a^{-1} b$ and $\beta := c^{-1} d$, given by 
\[
x = (b \plus d) (a \plus c)^{-1}.
\] 
\end{remark}


\subsection{Links between radicability and algebraic closedness}

An idempotent semifield is said to be \textit{order-dense} if, whenever $r < t$, there exists some $s$ with $r < s < t$. 
The following lemma extends Butkovi\v{c} \cite[Proposition~1.1]{Butkovic94}. 

\begin{lemma}[Butkovi\v{c}]\label{lem:orderdense}
If $\cis$ is radicable, 
then $\cis$ is order-dense. 
\end{lemma}

\begin{proof}
Let $s < t$. 
If $s \neq 0$, then $s < (s t)^{1/2} < t$. 
If $s = 0$ and $t = 1$, then there exists $r$ such that $s < r < t$ (recall the assumption $\cis \neq \mathbb{B}$ that we have made all along this paper). 
If $s = 0$ and $t < 1$, then $s < t^2 < t$. 
If $s = 0$, $t \neq 1$, and $t \not< 1$, then $s < t \wedge t^{-1} < t$. 
\end{proof}

The following result extends Rump \cite[Proposition~4]{Rump15}. 
We use the conventions $0^0 = 1$ in $\cis$ and $p_{-1} = 0$ for every $p \in \cis[X]$. 

\begin{theorem}\label{thm:radalg}
Every radicable idempotent semifield $\cis$ is order-dense and algebraically closed. 
Moreover, 
\begin{equation}\label{eq:pjbar2}
\overline{p}_j = \Plus_{k < j \leqslant \ell} (p_k^{\ell-j} p_{\ell}^{j-k})^{1/(\ell-k)}, 
\end{equation}
for all $p \in \cis[X]$ and $j \in \mathbb{N}$, where the indices $k, \ell$ run over $\mathbb{N} \cup \{ -1 \}$. 
\end{theorem}

\begin{proof}
By Lemma~\ref{lem:orderdense}, $\cis$ is order-dense. 
Let $p \in \cis[X]$ be a polynomial of degree $n$ and valuation $v$, and let $j \in \mathbb{N}$. 
%
If $j < v$ or $j > n$, then $\overline{p}_j = 0$ by Lemma~\ref{lem:0n}, and Equation~\eqref{eq:pjbar2} follows. 
If $j = v$, then $\overline{p}_v = p_v^{}$ by Lemma~\ref{lem:0n}, and Equation~\eqref{eq:pjbar2} also holds. 
Now assume that $v < j \leqslant n$. 
For all $x \in \cis^*$, we have 
\[
\widehat{p}(x) x^{-j} = \widehat{p^{(j)}}(x) \plus \widehat{q}(x^{-1}), 
\]
where $p^{(j)}$ is the $j$th derivative of $p$ and $q$ is the polynomial defined by $q(X) := p_0 X^j \plus \ldots \plus p_{j-1} X$. 
From the assumption $v < j$, we deduce that $q \neq 0$. 
Thus, $\widehat{q}$ is a bijection by Lemma~\ref{lem:pinv}, and so is $\widehat{p^{(j)}}$. 
We show that $y_{j}$ is the inf of $A(p, j) := \{ \widehat{p}(x) x^{-j} : x \in \cis^* \}$, where $y_{j}$ is defined by 
\[
y_{j} := \Plus_{k < j \leqslant \ell} (p_k^{\ell-j} p_{\ell}^{j-k})^{1/(\ell-k)}. 
\]
Note that $y_{j} \neq 0$ since $y_{j} \geqslant (p_v^{n-j} p_{n}^{j-v})^{1/(n-v)} > 0$. 
We first observe that, for every $y \in \cis^*$, 
\begin{align*}
y \geqslant y_{j} &\Leftrightarrow y \geqslant (p_k^{\ell-j} p_{\ell}^{j-k})^{1/(\ell-k)} \quad \mbox{ for all } k < j \leqslant \ell \\
&\Leftrightarrow y \geqslant p_{\ell}^{} (p_k^{1/(j-k)} y_{}^{-1/(j-k)})^{\ell-j} \quad \mbox{ for all } k < j \leqslant \ell \\
&\Leftrightarrow y \geqslant \Plus_{k < j} p_{\ell}^{} \left(p_k^{1/(j-k)} y_{}^{-1/(j-k)}\right)^{\ell-j} \quad \mbox{ for all } \ell \geqslant j. 
\end{align*}
By the Frobenius identity of Lemma~\ref{lem:frobenius}, we obtain
\begin{align*}
y \geqslant y_{j} &\Leftrightarrow y \geqslant \Plus_{\ell \geqslant j} p_{\ell}^{} \left(\Plus_{k < j} p_k^{1/(j-k)} y_{}^{-1/(j-k)}\right)^{\ell-j} \\
&\Leftrightarrow y \geqslant \Plus_{\ell \geqslant j} p_{\ell}^{} \left(\Plus_{m > 0} (q_m y_{}^{-1})^{1/m}\right)^{\ell-j} \mbox{ with } m := j - k, 
\end{align*}
for every $y \in \cis^*$. 
Since $y \neq 0$, we have $\widehat{q}^{\#}(y) \neq 0$ by the second part of Lemma~\ref{lem:pres}, 
so we can take $x := \widehat{q}^{\#}(y)^{-1}$. 
Using Equation~\eqref{eq:psharp} applied to $\widehat{q}$, this leads to
\[
y \geqslant y_{j} \Leftrightarrow y \geqslant \Plus_{\ell \geqslant j} p_{\ell} x^{\ell-j} = \widehat{p^{(j)}}(x), 
\]
for every $y \in \cis^*$. 
We deduce that $y_j$ is a lower bound of $A(p, j)$. 
Moreover, the previous equivalence gives in particular $y_{j} \geqslant \widehat{p^{(j)}}(x_{j})$, where $x_{j} := \widehat{q}^{\#}(y_{j})^{-1}$. 
Thus, 
\[
y_{j}^{} = \widehat{p^{(j)}}(x_{j}^{}) \plus y_{j}^{} = \widehat{p^{(j)}}(x_{j}^{}) \plus \widehat{q}(x_{j}^{-1}) = \widehat{p}(x_{j}^{}) x_{j}^{-j}. 
\]
This shows the inf of $A(p, j)$ is reached at $x_j$ and that $y_{j}$ is the inf of $A(p, j)$, as required. 
\end{proof}


Here is a converse statement to the previous result in the case where the idempotent semifield is totally ordered. 

\begin{corollary}\label{coro:eqalg}
Assume that $\cis$ is totally ordered. 
Then the following conditions are equivalent: 
\begin{enumerate}
	\item\label{eqalg1} $\cis$ is order-dense algebraically closed; 
	\item\label{eqalg2} $\cis$ is order-dense and preradicable; 
	\item\label{eqalg3} $\cis$ is radicable. 
\end{enumerate}
\end{corollary}


\begin{proof}
\eqref{eqalg3} $\Rightarrow$ \eqref{eqalg1} $\Rightarrow$ \eqref{eqalg2} follows from Theorem~\ref{thm:radalg} and Proposition~\ref{prop:prerad}. 

\eqref{eqalg2} $\Rightarrow$ \eqref{eqalg3}. 
Assume that $\cis$ is a totally ordered, order-dense, preradicable idempotent semifield. 
Let $t \in \cis$, $n \in \mathbb{N}^*$. 
We can suppose that $t \neq 0$ without loss of generality. 
Since $\cis$ is preradicable, there exists a least element $r_1$ such that $r_1^n \geqslant t$, and a least $r_2$ such that $r_2^n \geqslant t^{-1}$. Thus, $r_1 \geqslant r_2^{-1}$ by Lemma~\ref{prop:perfore}. 
If $r_1 = r_2^{-1}$, then this common value is the $n$-root of $t$ we are looking for. 
Now assume that $r_1 > r_2^{-1}$. 
Since $\cis$ is order-dense, there is some $s \in \cis$ such that $r_1 > s > r_2^{-1}$. Since $\cis$ is totally ordered we have $r_1^n > s^n \geqslant t$ or $r_2^n > s^{-n} \geqslant t^{-1}$; the former case contradicts the definition of $r_1$, and the latter case the definition of $r_2$. 
This shows that $\cis$ is radicable. 
\end{proof}




\begin{remark}
We have seen that $\mathbb{Z}_{\max}$, as a complete  idempotent semifield, is algebraically closed (see Corollary~\ref{coro:zmax}), hence preradicable. 
Note however that it is not order-dense (and not radicable). 
\end{remark}

\begin{corollary}
The following totally ordered idempotent semifields: 
\begin{itemize}
	\item $\mathbb{Q}_{\max} = (\mathbb{Q} \cup \{ -\infty\}, \max, -\infty, +, 0)$, 
	\item $\mathbb{A}_{\max} = (\mathbb{A} \cup \{ -\infty\}, \max, -\infty, +, 0)$, where $\mathbb{A}$ is the set of real algebraic numbers, 
\end{itemize}
are algebraically closed. 
\end{corollary}

\begin{remark}
Grigg proved that $\mathbb{Q}_{\max}$ is algebraically closed (see \cite[Theorem~5]{Grigg07}). 
\end{remark}



%

\subsection{Radicability and rational polynomials}

Let $p, q$ be polynomials in $\cis[X]$. 
We say that $p$ and $q$ are \textit{$\cis$-congruent} if there exists some non-zero polynomial $r \in \cis[X]$ such that $r(X) p(X) = r(X) q(X)$. 
This is an equivalence relation, and the quotient set $\cis\{X\}$ can be equipped with the structure of a $\times$-cancellative idempotent semiring. 
The equivalent classes with respect to $\cis$-congruence, i.e.\ the elements of $\cis\{X\}$, are the \textit{rational polynomials} over $\cis$, as per the terminology coined by Castella \cite{Castella13}. 

\begin{remark}
Rump \cite{Rump15} used a different construction. 
He first considered $\cis(X)$, the idempotent semifield of \textit{rational fractions} over $\cis$, defined as the Grothendieck  group (which is a lattice-ordered group) obtained from the monoid $(\cis[X] \setminus \{ 0 \}, \times, 1)$, with a zero adjoined. 
Then $\cis\{X\}$ was itself defined as the image of $\cis[X]$ by the natural morphism $\cis[X] \to \cis(X)$. 
\end{remark}

\begin{lemma}\label{lem:sum}
Let $p, q  \in \cis[X]$. 
Then $(p(X) \plus q(X))^m$ and $p(X)^m \plus q(X)^m$ are $\cis$-congruent, for all $m \in \mathbb{N}$. 
\end{lemma}

\begin{proof}
Lemma~\ref{lem:frobenius} applies since $\cis\{X\}$ is $\times$-cancellative. 
\end{proof}

\begin{lemma}\label{lem:equiv}
Let $p, q  \in \cis[X]$. 
If $p(X)^m$ and $q(X)^m$ are $\cis$-congruent for some $m \in \mathbb{N}^* 	$, 
then $p(X)$ and $q(X)$ are $\cis$-congruent. 
\end{lemma}

\begin{proof}
Lemma~\ref{prop:perfore} applies since $\cis\{X\}$ is $\times$-cancellative. 
\end{proof}

\begin{lemma}\label{lem:congrclosed}
Assume that $\cis$ is radicable. 
Then $p(X)$ and $\overline{p}(X)$ are $\cis$-congruent, for all $p \in \cis[X]$. 
\end{lemma}

\begin{proof}
Recall that $\overline{p}(X)$ exists by Theorem~\ref{thm:radalg} and satisfies  
\[
\overline{p}_j = \Plus_{k < j \leqslant \ell} (p_k^{\ell-j} p_{\ell}^{j-k})^{1/(\ell-k)}, 
\]
for all $j \in \mathbb{N}$. 
Take $n := \deg p$ and $m := n!$. 
Then 
\begin{align*}
(\overline{p}_j X^{j})^{m} &= \Plus_{k < j \leqslant \ell} (p_k X^k)^{(\ell-j) m / (\ell-k)} (p_{\ell} X^{\ell})^{(j-k) m / (\ell-k)} \\
&\leqslant \Plus_{k < j \leqslant \ell} p(X)^{(\ell-j) m / (\ell-k)} p(X)^{(j-k) m / (\ell-k)} \\
&\leqslant p(X)^{m}. 
\end{align*}
This shows that $\Plus_{j=0}^n (\overline{p}_j X^{j})^{m} \leqslant p(X)^{m}$, so by Lemma~\ref{lem:sum} $p(X)^{m} = p(X)^{m} \plus \Plus_{j=0}^n (\overline{p}_j X^{j})^{m}$ is $\cis$-congruent to 
\[
(p(X) \plus \Plus_{j=0}^n \overline{p}_j X^{j})^m = (p(X) \plus \overline{p}(X))^m = \overline{p}(X)^m. 
\]
By Lemma~\ref{lem:equiv}, $p(X)$ and $\overline{p}(X)$ are $\cis$-congruent. 
\end{proof}

The following result generalizes Castella \cite[Theorem~5.5]{Castella13} to the case of idempotent semifields that are not necessarily totally ordered. 

\begin{proposition}\label{prop:corresp}
In a radicable idempotent semifield, there is an isomorphism of semirings between rational polynomials and polynomial functions. 
\end{proposition}

\begin{proof}
Let $\cis$ be radicable, and let $p, q \in \cis[X]$. 
If $r(X) p(X) = r(X) q(X)$ for some non-zero polynomial $r \in \cis[X]$, we can suppose without loss of generality that $\widehat{r}(0) \neq 0$ by dividing by $X^{\val r}$. 
We get $\widehat{r}(x) \widehat{p}(x) = \widehat{r}(x) \widehat{q}(x)$ and $\widehat{r}(x) \neq 0$ for all $x \in \cis$, so that $\widehat{p}(x) = \widehat{q}(x)$ for all $x \in \cis$. 

Conversely, assume that $\widehat{p}(x) = \widehat{q}(x)$ for all $x \in \cis$. 
By the previous lemma there are some non-zero polynomials $r, s \in  \cis[X]$ such that $r(X) p(X) = r(X) \overline{p}(X)$ and $s(X) q(X) = s(X) \overline{q}(X)$. 
Since $\widehat{p}(x) = \widehat{q}(x)$ for all $x \in \cis$, we have $\overline{p}(X) = \overline{q}(X)$, so we obtain $t(X) p(X) = t(X) q(X)$ with $t(X) := r(X) s(X)$, which is non-zero. 
It follows that $p(X)$ and $q(X)$ are $\cis$-congruent. 

It is then licit to consider the map $\Phi$ that associates to a polynomial function $\widehat{p}$ with $p \in \cis[X]$ the rational polynomial 
\[
\left\{ q \in \cis[X] : q(X) \mbox{ $\cis$-congruent to } p(X) \right\} \in \cis\{X\}. 
\]
The first part of the proof shows that $\Phi$ is injective. 
Moreover, $\Phi$ is obviously a surjective morphism of semirings, so it is the isomorphism we were aiming at. 
\end{proof}

The following result extends Rump \cite[Corollary~2]{Rump15}. 
It shows that being radicable for an idempotent semifield is equivalent to being ``algebraically closed'' for rational polynomials over it.  

\begin{theorem}\label{thm:rat}
The following conditions are equivalent:
\begin{enumerate}
  \item\label{thm:rat1} $\cis$ is radicable;
  \item\label{thm:rat2} every rational polynomial in $\cis\{X\}$ splits in $\cis$;
  \item\label{thm:rat3} every polynomial in $\cis[X]$ is $\cis$-congruent to a closed polynomial. 
\end{enumerate}
\end{theorem}

\begin{proof}
\eqref{thm:rat1} $\Rightarrow$ \eqref{thm:rat3} is Lemma~\ref{lem:congrclosed}. 

\eqref{thm:rat2} $\Leftrightarrow$ \eqref{thm:rat3} is a consequence of Theorem~\ref{thm:ft1}. 

\eqref{thm:rat3} $\Rightarrow$ \eqref{thm:rat1}. 
Assume that every $p \in \cis[X]$ is $\cis$-congruent to a closed polynomial. 
Let $n \in \mathbb{N}^*$, $t \in \cis^*$, and take $p(X) := X^n \plus t$. 
Then $p(X)$ is $\cis$-congruent to a polynomial of the form $(X \plus c_1) \cdots (X \plus c_n)$ with $c_1, \ldots, c_n \in \cis$ and $c_1 \geqslant \cdots \geqslant c_n$. 
Considering the radicable closure $\widetilde{\cis}$ of $\cis$ (see Theorem~\ref{thm:radclo} in the Appendix), we have in particular 
\[
\widehat{p}(x) = (x \plus c_1) \cdots (x \plus c_n), 
\]
for all $x \in \widetilde{\cis}$. 
Now, from the proof of Proposition~\ref{prop:prerad}, we know that $c_1$ is the least $s \in \widetilde{\cis}$ such that $t \leqslant s^{n}$. 
But in $\widetilde{\cis}$, $t^{1/n}$ has the same property. 
Thus, $c_1 = t^{1/n}$, so that $t = c_1^{n}$. 
This shows that $\cis$ is radicable. 
\end{proof}

\section{Conclusion and perspectives}

As a next step, we shall investigate generalizations to polynomials with coefficients in a \textit{strict semifield} (a semifield that is not a field). 
We also aim to study further \textit{extensions} of (idempotent) semifields, after Castella \cite[Section~5]{Castella13}, or Tolliver \cite{Tolliver14} who examined finite extensions of $\mathbb{Z}_{\max}$. 
An additional perspective is to clarify the various notions of \textit{algebraic elements} that can be considered and the links between them. 


\begin{acknowledgements*}
I would like to thank Prof.\ Dominique Castella for the numerous exchanges we had and for his insightful remarks. 
\end{acknowledgements*}

\appendix

\section{Radicable closure of idempotent semifields}\label{sec:radclo}

We show that every idempotent semifield $\cis$ can be embedded into a radicable idempotent semifield $\widetilde{\cis}$, which we call the \textit{radicable closure} of $\cis$. 
To this aim, we define on $\mathbb{N}^* \times \cis$ the following relation: 
\[
(n, x) \sim (m, y) \Longleftrightarrow x^{m} = y^{n}. 
\]
This is an equivalence relation; transitivity is a consequence of the implication $(\exists n \in \mathbb{N}^* : x^n = y^n) \Rightarrow (x = y)$, see Lemma~\ref{prop:perfore}. 
The equivalence class of $(n, x)$ is written $(n, x)^{\sim}$, and we denote by $\widetilde{\cis}$ the associated quotient set. 
It is possible to endow $\widetilde{\cis}$ with the structure of an idempotent semifield by 
\[
(n, x)^{\sim} \plus (m, y)^{\sim} = (n m, x^m \plus y^n)^{\sim}
\]
and 
\[
(n, x)^{\sim} \cdot (m, y)^{\sim} = (n m, x^m y^n)^{\sim}.  
\]
One can prove that these definitions do not depend on the choice of the representing elements of the equivalence classes; in this process commutativity of $\cis$ is used. 
We omit the details, since the proofs are very similar to the construction of local rings $S^{-1} R$ in classical commutative algebra. 
The absorbing element of $\widetilde{\cis}$ is $0 := (1, 0)^{\sim}$ and its identity element is $1 := (1, 1)^{\sim}$. 
The $n$th power of an element $(m, x)^{\sim}$ equals $(m, x^n)^{\sim}$, which implies that $\widetilde{\cis}$ is indeed radicable. 
The map $i_{\cis} : \cis \rightarrow \widetilde{\cis}, x \rightarrow (1, x)^{\sim}$ is an injective semifield morphism that embeds $\cis$ into $\widetilde{\cis}$. 
We summarize these results with the following theorem. 

\begin{theorem}\label{thm:radclo}
Every idempotent semifield $\cis$ can be algebraically embedded into a radicable idempotent semifield $\widetilde{\cis}$ that has the following universal property: 
\[
\begin{tikzcd}
\cis \arrow{r}{f} \arrow[swap]{d}{i_{\cis}} & N \\
\widetilde{\cis} \arrow[dashed]{ur}[swap]{\widetilde{f}} &
\end{tikzcd}
\]
in the sense that every semifield morphism $f : \cis \to N$, with $N$ radicable, uniquely extends to a semifield morphism $\widetilde{f} : \widetilde{\cis} \to N$ such that $f = \widetilde{f} \circ i_{\cis}$,  
where $i_{\cis} : \cis \rightarrow \widetilde{\cis}$ is the injective semifield morphism that embeds $\cis$ into $\widetilde{\cis}$. 
Moreover, $\cis$ is totally ordered if and only if $\widetilde{\cis}$ is totally ordered. 
\end{theorem}


\bibliographystyle{plain}

\def\cprime{$'$} \def\cprime{$'$} \def\cprime{$'$} \def\cprime{$'$}
  \def\ocirc#1{\ifmmode\setbox0=\hbox{$#1$}\dimen0=\ht0 \advance\dimen0
  by1pt\rlap{\hbox to\wd0{\hss\raise\dimen0
  \hbox{\hskip.2em$\scriptscriptstyle\circ$}\hss}}#1\else {\accent"17 #1}\fi}
  \def\ocirc#1{\ifmmode\setbox0=\hbox{$#1$}\dimen0=\ht0 \advance\dimen0
  by1pt\rlap{\hbox to\wd0{\hss\raise\dimen0
  \hbox{\hskip.2em$\scriptscriptstyle\circ$}\hss}}#1\else {\accent"17 #1}\fi}

\end{document}